\documentclass[11pt,reqno]{amsart}
\usepackage{a4wide}
\usepackage{amssymb}
\usepackage{amsthm}
\usepackage{amsmath,todonotes}
\usepackage{amscd}
\usepackage{cancel}
\usepackage{verbatim}
\usepackage{color}
\usepackage[all]{xy}
\usepackage[mathscr]{eucal}
\usepackage{mathrsfs}
\usepackage{fullpage}
\usepackage{enumitem}
\usepackage{hyperref}
\usepackage{bbm}
\usepackage{qtree}
\allowdisplaybreaks

\numberwithin{equation}{section}

\newcommand{\SO}{\operatorname{SO}}
\renewcommand{\O}{\operatorname{O}}

\newtheorem{theorem}{Theorem}
\newtheorem{thm}[theorem]{Theorem}
\newtheorem{lemma}[theorem]{Lemma}
\newtheorem{lem}[theorem]{Lemma}

\newtheorem{proposition}[theorem]{Proposition}

\theoremstyle{remark}

\newtheorem*{remark}{Remark}
\newtheorem*{remarks}{Remarks}
\theoremstyle{definition}

\numberwithin{theorem}{section} \numberwithin{equation}{section}

\newcommand{\ord}{\text {\rm ord}}

\newcommand{\R}{\mathbb{R}}
\newcommand{\C}{\mathbb{C}}

\newcommand{\Q}{\mathbb{Q}}

\newcommand{\Z}{\mathbb{Z}}
\newcommand{\N}{\mathbb{N}}
\newcommand{\SL}{{\text {\rm SL}}}
\newcommand{\GL}{{\text {\rm GL}}}

\newcommand{\re}{\textnormal{Re}}
\newcommand{\im}{\textnormal{Im}}

\def\H{\mathbb{H}}

\begin{document}
\title[Universal sums of $m$-gonal numbers]{Universal sums of $m$-gonal numbers}
\author{Ben Kane}
\address{\rm Mathematics Department, University of Hong Kong, Pokfulam, Hong Kong}
\email{bkane@hku.hk}
\author{Jingbo Liu}
\address{\rm Mathematics Department, University of Hong Kong, Pokfulam, Hong Kong}
\email{jliu02@hku.hk}
\address{Current Address: \rm Mathematics Department, Texas A\&M University, San Antonio, Texas, 78224, USA}
\email{jliu@tamusa.edu}
\date{\today}


\medskip
\begin{abstract}
In this paper we study universal quadratic polynomials which arise as sums of polygonal numbers.  Specifically, we determine an asymptotic upper bound (as a function of $m$) on the size of the set $S_m\subset\N$ such that if a sum of $m$-gonal numbers represents $S_m$, then it represents $\N$.
\end{abstract}

\thanks{  The research of the first author was supported by grants from the Research Grants Council of the Hong Kong SAR, China (project numbers HKU 27300314, 17302515, 17316416, 17301317, and 17303618). Most of the research was conducted while the second author was a postdoctoral fellow at The University of Hong Kong.
}
\maketitle

\section{Introduction And Statement Of Results}\label{sec:intro}
The Conway--Schneeberger Fifteen theorem states that a given positive definite integral quadratic form is \begin{it}universal\end{it} (i.e., represents every positive integer) if and only if it represents the integers up to $15$ (a smaller subset of these numbers actually suffices).  In particular, the sums of squares
\[
Q(x)=\sum_{j=1}^{n} a_j x_j^2
\]
are universal if and only if they represent every integer up to $15$.  It was shown in \cite{BosmaKane} that a sum
\[
f(x)=\sum_{j=1}^n a_j T_{x_j}
\]
of triangular numbers $T_{x_j}:=\frac{x_j(x_j+1)}{2}$ is universal if and only if it represents every integer up to $8$.  In this paper, we are interested in generalizing this question to consider sums
\begin{equation}\label{eqn:sumsmgonal}
f(x)=\sum_{j=1}^{n} a_j P_{m}(x_j)
\end{equation}
of \begin{it}(generalized) $m$-gonal numbers\end{it}
\[
P_m(x_j):=\frac{(m-2)x_j^2 - (m-4)x_j}{2},
\]
where these coefficients $a_j$s are positive integers and these $x_j$s are chosen from $\mathbb{Z}$.

Constructing the possible universal sums of generalized $m$-gonal numbers by using an escalator tree method of Bhargava, one can see that there exists a (unique, minimal) $\gamma_m\in\N$ such that if every positive integer less than or equal to $\gamma_m$ is represented by $f$, then $f$ is universal. This is because there are only finitely many nodes appearing on the tree and $\gamma_m$ is actually the largest value of the truants (the smallest natural number not represented) of these nodes; see Lemma \ref{lem:finitegammam} for more details. Bosma and the first author \cite{BosmaKane} have shown that $\gamma_3=8$. When $m=4$, we can deduce that $\gamma_4=15$ from the Conway--Schneeberger Fifteen Theorem and the fact that the quaternary quadratic form $x^2+2x^2+5x^2+5x^2$ represents every positive integer except for $15$. For $m=5$, Ju \cite{Ju} recently showed that $\gamma_5=109$. Since each hexagonal number can be written as a triangular number and vice versa, we conclude for the $m=6$ case that $\gamma_6=\gamma_3=8$. The $m=8$ case has been resolved by Ju and Oh \cite{JuOh}, who proved that $\gamma_8=60$. Having established a number of individual cases, it is then natural to ask about the growth of $\gamma_m$ as a function of $m$.
\begin{theorem}\label{thm:main}
\begin{enumerate}[leftmargin=*,label={\rm(\arabic*)}]
\item
For $m\geq 3$ and every $\varepsilon>0$, there exists an absolute (effective) constant $C_{\varepsilon}$ such that
\[
\gamma_m\leq C_{\varepsilon}m^{7+\varepsilon}.
\]
\item There is no uniform upper bound which holds for all $m$. Specifically, if $m\geq 6$, then
\[
\gamma_m\geq m-4
\]
and for every element $\ell\in\N$ there exists a sum of generalized polygonal numbers which represents every nonnegative integer except for $\ell$.
\end{enumerate}
\end{theorem}
\begin{remark}
Theorem \ref{thm:main} (2) is due to Guy \cite{Guy} and is by explicit construction for the form with $a_j=1$, but we include it here for comparison with the upper bound obtained in Theorem \ref{thm:main} (1).
\end{remark}

It may be natural to ask whether one can consider similar questions with mixed sums of $m_j$-gonal numbers.
\begin{equation}\label{eqn:mixed}
f(x)=\sum_{j=1}^{n} a_j P_{m_j}(x_j).
\end{equation}
Theorem \ref{thm:main} (2) immediately implies that there is no uniform bound for universality of arbitrary mixed sums of the type \eqref{eqn:mixed}, but one may add the restriction $m_j<M$ in \eqref{eqn:mixed} for some fixed $M$ and ask whether there is a bound $b_M$ such that any sum of the type \eqref{eqn:mixed} with $m_j<M$ is universal if and only if it represents every integer up to $b_M$. Techniques similar to those used in the proof of Theorem \ref{thm:main} (1) should lead to the existence of such a bound, but a quantitative version of this bound would involve more careful calculations than those used in the proof of Theorem \ref{thm:main} (1). This may be an interesting direction for future research.

The proof of Theorem \ref{thm:main} goes through the theory of modular forms.  Roughly speaking, one constructs Bhargava's escalator tree up to a fixed depth $n_0\geq 4$ (i.e., one has $n=n_0\geq 4$) and then splits the corresponding theta series into an Eisenstein series component and a cusp form component.  Then one obtains an upper bound for the coefficients of the cuspidal part and a lower bound for the coefficients of the Eisenstein series component.  The Eisenstein series yields the main asymptotic term and is positive whenever a number is locally represented; hence the $h$th coeff of the theta series is positive for $h$ sufficiently large.   We give a lower bound for the Eisenstein series part in Section \ref{sec:Eisenstein}, and bound the coefficients of the cuspidal part from above in Section \ref{sec:cusp}. Although these bounds depend on the specific lattice $L$, it turns out that we can ignore this influence because the lattices occurring in the calculation of the upper bound for $\gamma_m$ are independent of $m$ for $m$ sufficiently large, and hence there are only finitely many lattices involved. More details can be found in the proof of Theorem \ref{thm:main} in Section \ref{sec:universal}.

\section{Preliminaries}\label{sec:prelim}
Before obtaining a bound for $\gamma_m$, we prove its existence.

\begin{lemma}\label{lem:finitegammam}
The constant $\gamma_m$ exists and is finite.
\end{lemma}
\begin{proof}
In order to show that $\gamma_m$ is finite, we first describe the construction of an object known as a \begin{it}escalator tree\end{it}, which was introduced by Bhargava in his proof of the Conway--Schneeberger Fifteen Theorem. For convenience we use the $n$-tuple $[a_1,...,a_n]$ to denote the sum of generalized $m$-gonal numbers $\sum_{j=1}^n a_jP_m(x_j)$, and without loss of generality we assume that $a_1\leq \cdots\le a_n$. Suppose that $f=[a_1,...,a_n]$ is not universal. Then we define the truant of $f$ to be the smallest positive integer which is not represented by $f$, and define an escalation of $f$ to be any sum $g=[a_1,...,a_n,a_{n+1}]$ which represents the truant of $f$. We will then construct an escalator tree by forming an edge between $f$ and $g$, with $\emptyset$ as the root. If $g$ is universal, then $g$ will be a leaf of our tree.

We claim that there are only finitely many nodes in this escalator tree. Note that each node is either universal or has only finitely many escalations. When $n\geq 5$, Chan and Oh \cite[Theorem 4.9]{CO13} show that the local-to-global principle (i.e., that every number which is locally represented by a positive definite quadratic polynomial is globally represented) is valid for sufficiently large integers. Therefore, when $n=5$, each sum $f=[a_1,...,a_n]$ on the tree represents every integer except for finitely many sporadic exceptions and finitely many congruence classes. As there are only finitely many such $f$ and each escalation represents at least one more integer or one more congruence class than their parents, we will complete our construction in finitely many steps. If a sum $g$ of generalized $m$-gonal numbers represents every truant of nodes on the tree, then it must contain some leaf (universal sum) as a partial sum. Hence $g$ is universal. Therefore $\gamma_m$ is the largest truant of the nodes on the tree.
\end{proof}

In order to obtain a bound for $\gamma_m$, we employ both algebraic and analytic techniques from the theory of quadratic lattices and quadratic forms. For the analytic side, we introduce some basic definitions about modular forms. Let $\mathbb{H}=\{z\in\mathbb{C} : \mathrm{Im}(z)>0\}$ be the upper half plane. The matrix $\gamma=\left(\begin{smallmatrix} a&b\\ c&d\end{smallmatrix}\right)\in \GL_2(\R)$ acts on $\mathbb{H}$ via fractional linear transformations as $\gamma z:=\frac{az+b}{cz+d}$ and it is standard to write $j(\gamma,z):=cz+d$. Let $f(z):\mathbb{H}\to\mathbb{C}$ be a function. For each $k\in\mathbb{Z}$, we use the notation $f|_k\gamma$ to denote the function whose value at $z$ is
\[
f(z)|_k\gamma:=\det(\gamma)^{\frac{k}{2}} j(\gamma,z)^{-k}f(\gamma z).
\]
For $\Gamma\subseteq \SL_2(\Z)$, we call $\nu:\Gamma\to \C$ a \begin{it}multiplier system\end{it} for $\Gamma$ of weight $k$ if for every $\gamma,M\in\Gamma$
\[
\nu(\gamma M) j(\gamma M,z)^k  = \nu(\gamma)j(\gamma,Mz)^k\nu(M)j(M,z)^k.
\]
We also refer to the multiplier system as a \begin{it}Nebentypus character\end{it} if $\nu$ is a character. A (holomorphic) \begin{it}modular form\end{it} of weight $k\in\mathbb{Z}$ and multiplier system $\nu$ for a congruence subgroup $\Gamma(N)\subseteq\Gamma\subseteq \mathrm{SL}_2(\mathbb{Z})$ is a function $f$ which satisfies the following conditions:
\begin{enumerate}
\item $f$ is holomorphic on $\mathbb{H}$;
\item for every $\gamma\in\Gamma$, we have
\[
f|_k\gamma=\nu(\gamma) f;
\]
\item for any $\gamma_0\in\mathrm{SL}_2(\mathbb{Z})$,  $f(z)|_k\gamma_0$ has the form $\sum a_ne^{2\pi inz/N}$ with $a_n=0$ for all $n<0$. The number $N$ is called the \begin{it}cusp width\end{it}.
\end{enumerate}
Furthermore, if $a_0=0$ for all $\gamma_0\in\mathrm{SL}_2(\mathbb{Z})$, then we call $f$ a cusp form. The space of cusp forms is an inner product space with respect to the \begin{it}Petersson inner product\end{it}, defined for weight $k$ cusp forms $f$ and $g$ on $\Gamma$ with multiplier $\nu$ with $|\nu|=1$ by
\begin{equation}\label{eqn:innerdef}
\left<f,g\right>_{\Gamma}:=\int_{\Gamma\backslash\H} f(z)\overline{g(z)} y^k \frac{dx dy}{y^2}.
\end{equation}
We have included the subscript $\Gamma$ to emphasize that the definition depends on the choice of $\Gamma$ on the right-hand side. There is another normalization
\[
\left<f,g\right>:=\frac{\left<f,g\right>_{\Gamma}}{\left[\SL_2(\Z):\Gamma\right]}
\]
which is independent of the choice of $\Gamma$. Different authors use these two different normalizations, and both have advantages and disadvantages in different settings. One advantage of the normalization $\left<f,g\right>_{\Gamma}$ is given in Lemma \ref{lem:normbound} below. On the other hand, it is useful to note that the choice $\left<f,g\right>$ leads to an isometry with respect to the slash operator $|_k$. The following lemma is well-known, but we supply a proof for the convenience of the reader.
\begin{lemma}\label{lem:slashiso}
\noindent

\noindent
\begin{enumerate}[leftmargin=*,label={\rm(\arabic*)}]
\item
For $\gamma\in \GL_2(\Q)$ we have
\[
\left<f|_k \gamma,g|_k\gamma\right> =\left<f,g\right>.
\]
\item
In particular, defining $f(z)|V_N:=f(Nz)$, we have
\[
\left<f|V_N,g|V_N\right> = N^{-k}\left<f,g\right>
\]
and
\[
\left<f|V_N,g|V_N\right>_{\Gamma'} = \left[\Gamma:\Gamma'\right]N^{-k}\left<f,g\right>_{\Gamma},
\]
where $f,g$ are modular on $\Gamma$ and $f|V_N,g|V_N$ are modular on $\Gamma'$.

\end{enumerate}
\end{lemma}
\begin{remark}
If $\Gamma=\Gamma_0(M)$ is the congruence subgroup
\[
\Gamma_0(M):=\left\{\left(\begin{matrix} a&b \\ c &d\end{matrix}\right)\in\SL_2(\Z): M\mid c\right\},
\]
then we may choose $\Gamma'=\Gamma_0(MN)$. In this case (cf. \cite[Proposition 1.7]{OnoBook}) we have
\begin{equation}\label{eqn:Gamma0index}
\left[\Gamma_0(M):\Gamma_0(MN)\right]= N \prod_{\substack{p\mid N\\ p\nmid M}} \left(1+\frac{1}{p}\right)=N^{1+o(1)}.
\end{equation}
\end{remark}
\begin{proof}
(1) Suppose that $f,g$ are modular on $\Gamma$ and set $\Gamma_{\gamma}:=\Gamma\cap \gamma^{-1}\Gamma \gamma$. Since $\left<f,g\right>$ is independent of the group and $f|_{k}\gamma$ is modular on $\Gamma_{\gamma}$, we may choose the group $\Gamma_{\gamma}$ for both of the inner products. Noting that $\im(\gamma z)=\frac{y \det(\gamma)}{|cz+d|^2}$, we compute
\begin{align*}
\left<f|_k\gamma, g|_k\gamma\right> &= \frac{1}{\left[\SL_2(\Z):\Gamma_{\gamma}\right]}\int_{\Gamma_{\gamma}\backslash\H} \det(\gamma)^{k} f(\gamma z)\overline{g(\gamma z)} \frac{y^k}{|j(\gamma,z)|^{2k}}  \frac{dxdy}{y^2}\\
&=\frac{1}{\left[\SL_2(\Z):\Gamma_{\gamma}\right]}\int_{\Gamma_{\gamma}\backslash\H}  f(\gamma z)\overline{g(\gamma z)} \im(\gamma z)^k  \frac{dxdy}{y^2}\\
&=\frac{1}{\left[\SL_2(\Z):\Gamma_{\gamma}\right]}\int_{\Gamma_{\gamma}\backslash\H}  f(z)\overline{g(z)} y^k  \frac{dxdy}{y^2}=\left<f,g\right>.
\end{align*}
Note that in the last line we made the change of variables $z\mapsto \gamma^{-1}z$ and used the fact that if $\mathcal{F}$ is a fundamental domain for $\Gamma_{\gamma}$, then so is $\gamma\mathcal{F}$.

(2) This follows immediately from part (1) because
\[
f|V_N= N^{-\frac{k}{2}} f\Big|_k\left(\begin{matrix}N&0\\ 0&1\end{matrix}\right)
\]
and $\left[\SL_2(\Z):\Gamma'\right]=\left[\SL_2(\Z):\Gamma\right]\left[\Gamma:\Gamma'\right]$.
\end{proof}
It is natural to define the Petersson norm by $\|f\|^2:=\left<f,f\right>$ and we also denote $\|f\|_{\Gamma}^2:=\left<f,f\right>_{\Gamma}$.

On the algebraic side, we adopt the language of quadratic spaces and lattices. A quadratic space is a vector space $V$ with a symmetric bilinear form $b$ on it. A quadratic lattice $L$ on $V$ is a finitely generated $\mathbb{Z}$-module in $V$ with the property that $V=\mathbb{Q}L$. Then there is a basis $v_1,...,v_n$ of $V$ such that $L=\mathbb{Z}v_1+\cdots+\mathbb{Z}v_n$. With respect to this basis, we can associate a so-called Gram matrix $M_L=\left(b(v_i,v_j)\right)$ with $L$, and write $L\cong M_L$. When $M_L$ is a diagonal matrix with entries $a_1,...,a_n$ on the diagonal, it is written as $\langle a_1,...,a_n\rangle$. For a rational prime $p$, we define the localization of $V$ and $L$ by $V_p:=V\otimes_\mathbb{Q}\mathbb{Q}_p$ and $L_p:=L\otimes_\mathbb{Z}\mathbb{Z}_p$ respectively.
We also define the $p$-adic measure to be the unique translation-invariant measure such that $\int_{\Z_p} d\sigma=1$; this leads to the conclusion that $\int_{N\Z_p}d\sigma = p^{-\ord_p(N)}$.

By a shifted lattice we mean a coset $L+\nu$ where $\nu$ is a vector in $V$. Elements in this coset are of the form $x+\nu$ with $x\in L$. The smallest positive integer $N$ such that $N\nu\in L$ is called the conductor of the shifted lattice $L+\nu$. As in \eqref{eqn:sumsmgonal}, let $f(x)=\sum_{j=1}^n a_jP_m(x_j)$ be a sum of generalized $m$-gonal numbers.
 In order to investigate the nonnegative integers $\ell$ represented by $f$, for each such $\ell$ we write
\begin{equation}\label{eqn:hdef}
h=h(\ell):=\frac{2\ell}{m-2}+\sum_{j=1}^na_j\left(\frac{m-4}{2(m-2)}\right)^2.
\end{equation}
Let $X:=L+\nu$ where $L=\mathbb{Z}v_1+\cdots+\mathbb{Z}v_n\cong\langle a_1,...,a_n\rangle$ and $\nu=-\frac{m-4}{2(m-2)}(v_1+\cdots+v_n)\in V$. Then $\ell$ is represented by $f$ if and only if $h$ is represented by $X$. We call $X$ the corresponding shifted lattice of $f$. Due to equivalence between the representability of $\ell$ by each $f$ occurring in the escalator tree and the representability of $h(\ell)$ by the corresponding shifted lattice $X$, we investigate representations by shifted lattices in order to prove Theorem \ref{thm:main}. We use $\varphi(v)$ to denote the positive definite quadratic form associated with $V$, i.e., for $v=\sum_{j=1}^{n}\lambda_j v_j$ with $\lambda_j\in\Q$ we have
\begin{equation}\label{eqn:varphidef}
\varphi(v)=b(v,v)=\sum_{j=1}^{n} a_j\lambda_j^2.
\end{equation}

Let $\O(V)$ and $\SO(V)$ be the orthogonal group and the proper orthogonal group of $V$ respectively.  The class $\mathrm{cls}(X)$ and the proper class $\mathrm{cls}^+(X)$ of $X$ are defined as the orbits of $X$ under the action of $\O(V)$ and $\SO(V)$ respectively. The orbit of $X$ under the action of $\O_{\mathbb{A}}(V)$ is called the genus of $X$, and the orbit of $X$ under the action of $\SO_{\mathbb{A}}(V)$ is called the proper genus of $X$. We denote them by $\mathrm{gen}(X)$ and $\mathrm{gen}^+(X)$ respectively. By \cite[Lemma 4.2]{CO13}, all the elements in the $\mathrm{gen}(X)$ and $\mathrm{gen}^+(X)$ are shifted lattices on $V$. Suppose that $X_1,...,X_t$ represent the classes in the genus of $X$. Then we define
$$
R(h,X):=\sum_{i=1}^t\frac{r(h, X_i)}{|\O(X_i)|}\hspace{10mm} \mathrm{and}\hspace{10mm} m(X):=\sum_{i=1}^t|\O(X_i)|^{-1},
$$
where $r(h, X_i):=\#\{v\in X_i : \varphi(v)=h\}$ and $|\O(X_i)|$ is the cardinality of the orthogonal group $\O(X_i)$ of $X_i$. We call $m(X)$ the mass of $X$.

We put
\[
\hspace{12mm}\theta_X(z)=\sum_{v\in X}e^{2\pi i\varphi(v)z}=\sum_{h\in\mathbb{Q}}r(h,X)e^{2\pi ihz}:=\sum_{h\in\mathbb{Q}}a_{\theta_X}(h)e^{2\pi ihz}
\]
 and
\[
E_X(z)=\sum_{h\in\mathbb{Q}}R(h,X)m(X)^{-1}e^{2\pi ihz}:=\sum_{h\in\mathbb{Q}}a_{E_X}(h)e^{2\pi ihz}.
\]
Then, as explained by Shimura in \cite{Shimura2004}, $E_X(z)$ is an Eisenstein series of weight $n/2$. This means that for a positive rational number $h$, the coefficient $a_{E_X}(h)$ can be expressed as the product of local densities. This result in the case of quadratic forms was first proved by Siegel \cite{Siegel} and then generalized by Weil \cite{Weil}. Furthermore, Shimura \cite{Shimura2004} shows that the difference
\[
G_X(z)=\theta_X(z)-E_X(z):=\sum_{h\in\mathbb{Q}}a_{G_X}(h)e^{2\pi ihz}
\]
is a cusp form. In the following sections we are going to find an upper bound for $|a_{G_X}(h)|$ and a lower bound for $a_{E_X}(h)$. Then we can determine when $h$ is represented by $X$, i.e., $r(h,X)>0$.

\vskip20pt

\section{The Eisenstein Series Component}\label{sec:Eisenstein}
In this section, we give a lower bound and an upper bound on the coefficients of the Eisenstein series component.  Specifically, we bound the local density of a shifted lattice $X=L+\nu$ at each prime $p$, using a formula of Shimura \cite{Shimura2004} and formulas of Yang \cite{Yang}. In the following calculation, we assume more generally that

\begin{enumerate}
\item $L$ is a primitive positive definite integral lattice with even rank, and the Gram matrix $M_L=\langle a_1,...,a_n\rangle$ with respect to the basis $v_1,...,v_n$;
\vskip3pt
\item $\nu=-\frac{c}{N}(v_1+\cdots+v_n)$ with $c, N\in\mathbb{Z}$ and $(c,N)=1$, and hence $N$ is the conductor of $L+\nu$;
\vskip2pt
\item $h\in\Q$ satisfies the condition $h-\sum_{j=1}^na_j\left(c/N\right)^2\in 8((N,4)N)^{-1}\mathbb{Z}$.
\end{enumerate}

For any rational prime $p$ and $z\in\mathbb{Q}_p$, we define $\mathbf{e}_p(z)=\mathbf{e}(-y)=e^{2\pi i(-y)}$ with $y\in\bigcup_{t=1}^\infty p^{-t}\mathbb{Z}$ such that $z-y\in\mathbb{Z}_p$. Let $\lambda(v)$ and $\lambda_p(v)$ be the characteristic functions of $X$ and $X_p$ respectively. Normalizing the measures $dv$ and $d\sigma$ on $L_p$ and $\Z_p$ so that $\int_{L_p}dv=\int_{\mathbb{Z}_p}d\sigma=1$, the local density at $p$ is defined as
\[
b_p(h,\lambda,0):=\int_{\mathbb{Q}_p}\int_{V_p}\mathbf{e}_p(\sigma(\varphi(v)-h))\lambda_p(v)dvd\sigma.
\]
For any $\sigma\in\mathbb{Q}_p$, we define $\tau_p(\sigma)$ by
\[
\tau_p(\sigma)=\tau_{p,N,c}(\sigma):=\int_{\mathbb{Z}_p}\mathbf{e}_p(\sigma[Nx^2-2cx])dx,
\]
where the measure on $\mathbb{Z}_p$ is such that $\mathbb{Z}_p$ has measure 1.
Here $(N,c)=1$ and we omit the dependence on $N$ and $c$ when they are clear from the context; in particular, in the following calculations they are the $N$ and $c$ determining $\nu$ in condition (2) above.

\begin{lem}
Let $\chi$ be a real Dirichlet character. Then for any integer $s\geq 2$,
\[
\zeta(s)^{-1}\leq L(s,\chi)\leq \zeta(s).
\]
\end{lem}

\begin{proof}
The second inequality is obvious. We only need to prove the first one.
\[
\log L(s,\chi)=-\sum_p\log(1-\chi(p)p^{-s})=\sum_p\sum_{n=1}^\infty\frac{\chi(p^n)}{np^{ns}}    \geq-\sum_p\sum_{n=1}^\infty\frac{1}{np^{ns}}=\displaystyle-\log\zeta(s).
\]
\end{proof}

Then by the formula given in \cite[Theorem 1.5]{Shimura2004}, when $n\geq4$ is an even integer, we have the following inequalities:
\begin{multline}\label{eqn:ShimuraEis}
\frac{\pi^{\frac{n}{2}}h^{\frac{n-2}{2}}}{\Gamma(\frac{n}{2})\sqrt{\det(L)}\zeta(\frac{n}{2})^2}\cdot\prod_{p\mid e_1}b_p(h,\lambda,0)\prod_{p\mid he', p\nmid e_1}r_p\left(\frac{n}{2}\right)\leq\frac{R(h,X)}{m(X)}\\
\leq\frac{\pi^{\frac{n}{2}}\zeta(\frac{n}{2})^2h^{\frac{n-2}{2}}}{\Gamma(\frac{n}{2})\sqrt{\det(L)}}\cdot\prod_{p\mid e_1}b_p(h,\lambda,0)\prod_{p\mid he', p\nmid e_1}r_p\left(\frac{n}{2}\right),
\end{multline}
where $e'$ is the product of all finite primes $p$ at which the dual lattice $L_p^{\#}\neq 2L_p$, and $e_1$ is the product of all finite primes $p$ at which $h\notin\mathbb{Z}_p$ or $L_p$ is not maximal or $X_p\neq L_p$. The numbers $r_p(s)$ are given in \cite[Section 1.6]{Shimura2004} and one can check that when $n\geq 6$ is even, we have $1/2\leq r_p(n/2)\leq2$. Therefore, it suffices to bound values of $b_p(h,\lambda,0)$.

\subsection{\texorpdfstring{$\mathbf{p\geq 3}$}{p>=3} and \texorpdfstring{$\mathbf{p \mid N}$}{p divides N}}\label{subsec1}
\vskip20pt

\begin{lem}
Let $p$ be an odd prime divisor of $N$ and $c\in\Z$ with $(c,N)=1$. For any positive integer $t$, the map $$x\mapsto Nx^2-2cx$$ is a bijection of $\mathbb{Z}/p^t\mathbb{Z}$ onto itself.
\end{lem}

\begin{proof}
Since $\mathbb{Z}/p^t\mathbb{Z}$ is finite, it is enough to show that this map is injective. Suppose that $[Nx^2-2cx]-[Ny^2-2cy]\in p^t\mathbb{Z}$ with $x,y\in\mathbb{Z}$. Then $(x-y)[N(x+y)-2c]\in p^t\mathbb{Z}$. Since $(c,N)=1$ and $p$ is odd, $p\nmid2c$. Therefore $x-y\in p^t\mathbb{Z}$ and this map is injective.
\end{proof}

\begin{lem}\label{p divides N}
For any odd prime $p$ which divides $N$,
\[
\tau_p(\sigma)=
\begin{cases}
    1 & \text{if }\sigma\in \mathbb{Z}_p,\\
    0 & \text{if }\sigma\notin \mathbb{Z}_p.
\end{cases}
\]
\end{lem}
\begin{proof}
It is obvious that $\tau_p(\sigma)=1$ when $\sigma\in\mathbb{Z}_p$. Now suppose that $\sigma=p^{-t}\alpha$ with $\alpha\in\mathbb{Z}_p^\times$ and $t\geq1$. Then we have (noting that $\Z_p/p^t\Z_p\cong \Z/p^t\Z$)
\begin{align*}
\tau_p(p^{-t}\alpha)&=\sum_{x\in \Z_p/p^t\Z_p} \int_{p^t\Z_p}\mathbf{e}_p(p^{-t}\alpha[N(x+y)^2-2c(x+y)])dy
\\
&=\sum_{x\in \Z/p^t\Z} \mathbf{e}_p(p^{-t}\alpha[Nx^2-2cx])\int_{p^t\Z_p}dy
\\
&=p^{-t}\displaystyle\sum_{x\in\mathbb{Z}/p^t\mathbb{Z}}\mathbf{e}_p(p^{-t}\alpha[Nx^2-2cx])
\\
&=p^{-t}\displaystyle\sum_{y\in\mathbb{Z}/p^t\mathbb{Z}}\mathbf{e}_p(p^{-t}\alpha y)=0.
\end{align*}
\end{proof}
\begin{thm}\label{thm:pdivN}
Suppose that $p$ is an odd prime divisor of $N$. Then
\[
b_p(h,\lambda,0)=p^{-\mathrm{ord}_pN}.
\]
\end{thm}
\begin{proof}
We directly compute
\begin{align*}
b_p(h,\lambda,0)&=\displaystyle\int_{\mathbb{Q}_p}\int_{V_p}\mathbf{e}_p(\sigma(\varphi(v)-h))\lambda_p(v)dvd\sigma\\
&=\displaystyle\int_{\mathbb{Q}_p}\int_{L_p}\mathbf{e}_p(\sigma(\varphi(v+\nu)-h))dvd\sigma
\end{align*}

Writing $v=\sum_{j=1}^n \lambda_j v_j$ in the basis $v_1,\dots,v_n$, we then plug in (from the definition \eqref{eqn:varphidef})
\[
\varphi(v+\nu)= \sum_{j=1}^n a_j\left(\lambda_j -\frac{c}{N}\right)^2
\]
and use Lemma \ref{p divides N} and $(a_1,\dots,a_n)=1$ to obtain
\begin{align*}
\int_{\mathbb{Q}_p}\int_{L_p}\mathbf{e}_p(\sigma(\varphi(v+\nu)-h))dvd\sigma&=\displaystyle\int_{\mathbb{Q}_p}\mathbf{e}_p\left(\sigma\left(\sum_{j=1}^na_j(c/N)^2-h\right)\right)\prod_{j=1}^n\tau_p\left(\frac{a_j\sigma}N\right)d\sigma\\
&=\int_{N\Z_p}d\sigma= p^{-\mathrm{ord}_pN}.
\end{align*}
In the last line we used the fact that $\prod_{j=1}^n \tau_p\left(\frac{a_j\sigma}{N}\right)=0$ unless $\sigma\in \frac{N}{(a_1,\dots,a_n)}\Z_p=N\Z_p$, as the $a_1,\dots,a_n$ are relatively prime.
\end{proof}

\subsection{\texorpdfstring{$\mathbf{p\geq 3}$}{p>=3} and \texorpdfstring{$\mathbf{p \nmid N}$}{p doesn't divide N}}\label{subsec2}
\vskip20pt
Under this assumption, $X_p=L_p$ is an integral lattice over $\mathbb{Z}_p$. Thus we can use the formula introduced in \cite[Theorem 3.1]{Yang} to obtain bounds for local densities. Suppose that $L_p$ is equivalent to $\langle b_1 p^{r_1},\dots,b_n p^{r_n}\rangle$ with $b_i\in\mathbb{Z}_p^\times$ and $r_1\leq\dots\leq r_n$.

For each integer $t>0$, set
\[
L(t,1):=\{1\leq i\leq n: r_i-t<0\ \text{is odd}\} \hspace{20mm} l(t,1):=\#L(t,1).
\]
Furthermore, we define
\[
d(t):=\displaystyle t+\frac{1}{2}\sum_{r_i<t}(r_i-t), \hspace{20mm}\varepsilon(t):=\displaystyle\left(\frac{-1}{p}\right)^{[l(t,1)/2]}\prod_{i\in L(t,1)}\left(\frac{b_i}{p}\right).
\]
Then for $h=\alpha p^a$ with $\alpha\in\mathbb{Z}_p^\times$ and $a$ is a nonnegative integer, Yang \cite{Yang} has shown that
\[
b_p(h,\lambda,0)=1+R_1(1,h,L_p),
\]
where
\[
R_1(1,h,L_p):=(1-p^{-1})\sum_{\substack {0<t\leq a\\ l(t,1)\ \text{is even}}}\varepsilon(t)p^{d(t)}+\varepsilon(a+1)p^{d(a+1)}f_1(h).
\]
with
\[
f_1(h):=
\begin{cases}
    -\frac{1}{p}, & \text{if $l(a+1,1)$ is even,} \\
    \left(\frac{\alpha}{p}\right)\frac{1}{\sqrt{p}}, & \text{if $l(a+1,1)$ is odd.}
  \end{cases}
\]

We next determine the local density in cases where local representations are guaranteed; specifically, we restrict $r_1,\dots,r_4$ and $b_1,\dots, b_4$ in a way which guarantees that the lattice $L_p$ is (locally) universal and then bound the local density by a constant which only depends on $L$.

\begin{theorem}\label{p does not divide N}
Suppose that $n\geq6$. Let $p$ be an odd prime which does not divide $N$.  If $[r_1, r_2, r_3, r_4]$ is equal to one of the following
:
\begin{enumerate}[leftmargin=*,label={\rm(\arabic*)}]
\item $[0,0,0,i],\ i\geq 0$;
\item $[0,0,i,j],\ 1\leq i\leq j$ when $p\equiv1\pmod 4$ and $\left(\frac{b_1b_2}{p}\right)=1$ or $p\equiv3\pmod 4$ and $\left(\frac{b_1b_2}{p}\right)=-1$;
\item $[0,0,1,1]$ when $p\equiv1\pmod 4$ and $\left(\frac{b_1b_2}{p}\right)=-1$ or $p\equiv3\pmod 4$ and $\left(\frac{b_1b_2}{p}\right)=1$,
\end{enumerate}
then there are absolute positive constants $c_1(L)$ and $c_2(L)$ depending only on $L$ such that
\[
c_1(L)\leq b_p(h,\lambda,0)\leq c_2(L).
\]
\end{theorem}
\begin{proof}
For Case (1), note that when $t\leq r_4$ is odd, $l(t,1)=3$ is odd; when $r_4+1\leq t\leq r_n$, $d(t)\leq r_4/2-t$; and when $t\geq r_n+1$, we have $d(t)\leq r_4/2-t-1$ (noting that $n\geq 6$, so the sum defining $d(t)$ has at least $6$ summands).
Thus
\begin{align*}
\left|R_1(1,h,L_p)\right|&\leq(1-p^{-1})\left(p^{-1}+\cdots+p^{-\lfloor r_4/2\rfloor}+\sum_{t=r_4+1}^{\infty}p^{d(t)}\right)+p^{d(1)-1/2}\\
                       &\leq(1-p^{-1})\left(p^{-1}+\cdots+p^{-\lfloor r_4/2\rfloor}+\sum_{t=r_4+1}^{r_n}p^{r_4/2-t}+\sum_{t= r_n+1}^{\infty}p^{r_4/2-t-1}\right)+p^{-1}\\
                         &\leq2p^{-1}.
\end{align*}

For Case (2), when $a=0$, it is evident that $\left|R_1(1,h,L_p)\right|\leq p^{-1/2}$. Then note that when $t=1$, $l(1,1)=2$, $\varepsilon(1)=1$ and $d(1)=0$; when $2\leq t\leq r_4$, we see that $l(t,1)$ is either odd or $l(t,1)=2$ is even and both $d(t)=0$ and $\varepsilon(t)=1$; when $t\geq r_4+1$, we write $t=r_4+i$ with $i\geq 1$ and then note that $d(t)\leq (r_3+r_4)/2-t\leq -i$.
Thus we have
\[
1-2p^{-1}\leq(1-p^{-1})\left(1-\sum_{i=1}^\infty p^{-i}\right)\leq R_1(1,h,L_p)\leq(1-p^{-1})\left(r_4+\sum_{i=1}^\infty p^{-i}\right)\leq r_4.
\]

For Case (3), we have $d(t)\leq 1-t$ when $ 1\leq t\leq r_n$, and $d(t)\leq -t$ when $t\geq r_n+1$. Thus
\begin{align*}
\left|R_1(1,h,L_p)\right|&\leq(1-p^{-1})\sum_{t=1}^{\infty}p^{d(t)}\\
                         &\leq(1-p^{-1})\left(\sum_{t=1}^{r_n}p^{d(t)}+\sum_{t= r_n+1}^{\infty}p^{d(t)}\right)\\
                         &\leq 1-\left(p^{-r_n}-p^{-(r_n+1)}\right).
\end{align*}
Hence there exist absolute positive constants $c_1(L)$ and $c_2(L)$ which depend only on $L$ such that
\[
c_1(L)\leq b_p(h,\lambda,0)\leq c_2(L).
\]
\end{proof}

\subsection{\texorpdfstring{$\mathbf{p=2}$}{p equals 2} and \texorpdfstring{$\mathbf{2\| N}$}{2 exactly divides N }}\label{subsec3}
\vskip20pt
\begin{lem}
Suppose that $2\| N$. For any positive integer $t$, the map $$x\mapsto (N/2)x^2-cx$$ is a two-to-one surjection of $\mathbb{Z}/2^t\mathbb{Z}$ onto $2\mathbb{Z}/2^t\mathbb{Z}$.
\end{lem}
\begin{proof}
It suffices to show that each element of $2\mathbb{Z}/2^t\mathbb{Z}$ corresponds to at most two elements of $\mathbb{Z}/2^t\mathbb{Z}$. Suppose that $[(N/2)x^2-cx]-[(N/2)y^2-cy]\in 2^t\mathbb{Z}$ with $x,y\in\mathbb{Z}$. Then $(x-y)[(N/2)(x+y)-c]\in 2^t\mathbb{Z}$. If $x-y\in2\mathbb{Z}$, then $2\nmid (N/2)(x+y)-c$, so that $x-y\in 2^t\mathbb{Z}$. Thus we obtain that each element of $2\mathbb{Z}/2^t\mathbb{Z}$ corresponds to at most two elements of $\mathbb{Z}/2^t\mathbb{Z}$.
\end{proof}

\begin{lemma}\label{2 parallel N}
Suppose that $2\| N$. Then
\[
\tau_2(\sigma)=
\begin{cases}
    1, & \text{if }\sigma\in 2^{-2}\mathbb{Z}_p, \\
    0, & \text{if }\sigma\notin 2^{-2}\mathbb{Z}_p.
\end{cases}
\]
\end{lemma}
\begin{proof}
Suppose that $\sigma=2^{-t}\alpha$ with $\alpha\in\mathbb{Z}_2^\times$ and $t\geq3$. Then, following the same calculation as in the proof of Lemma \ref{p divides N}, we have
\[
\tau_2(2^{-t}\alpha)=2^{1-t}\displaystyle\sum_{x\in\mathbb{Z}/2^{t-1}\mathbb{Z}}\mathbf{e}_2(2^{1-t}\alpha[(N/2)x^2-cx])=2^{2-t}\displaystyle\sum_{y\in\mathbb{Z}/2^{t-2}\mathbb{Z}}\mathbf{e}_2(2^{2-t}\alpha y)=0.
\]
Therefore $\tau_2(\sigma)$ is $0$ or $1$ according to $\sigma\notin 2^{-2}\mathbb{Z}_2$ or $\sigma\in 2^{-2}\mathbb{Z}_2$.
\end{proof}

\vskip5pt
\begin{thm}
Suppose that $2\| N$. Then
\[
b_2(h,\lambda, 0)=2=2^{-(\mathrm{ord}_2N-2)}.
\]
\end{thm}
\begin{proof}
Following the calculation in the proof of Theorem \ref{thm:pdivN}, we compute
\begin{align*}
b_2(h,\lambda,0)&=\displaystyle\int_{\mathbb{Q}_2}\int_{V_2}\mathbf{e}_2(\sigma(\varphi(v)-h))\lambda_2(v)dvd\sigma\\
&=\displaystyle\int_{\mathbb{Q}_2}\int_{L_2}\mathbf{e}_2(\sigma(\varphi(v+\nu)-h))dvd\sigma\\
&=\displaystyle\int_{\mathbb{Q}_2}\mathbf{e}_2\left(\sigma\left(\sum_{j=1}^na_j(c/N)^2-h\right)\right)\prod_{j=1}^n\tau_2\left(\frac{a_j\sigma}{N}\right)d\sigma\\
&=\int_{ 2^{-1}\Z_2} d\sigma=2.
\end{align*}
\end{proof}

\subsection{\texorpdfstring{$\mathbf{p=2}$}{p equals 2} and \texorpdfstring{$\mathbf{4\mid N}$}{4 divides N}}\label{subsec4}
\vskip20pt
The argument here is similar to that in Subsection \ref{subsec1}. Nonetheless, we provide it here for the sake of completeness and clarity.

\begin{lem}
Suppose that $4\mid N$. For any positive integer $t$, the map $$x\mapsto (N/2)x^2-cx$$ is a bijection of $\mathbb{Z}/2^t\mathbb{Z}$ onto itself.
\end{lem}

\begin{proof}
Since $\mathbb{Z}/2^t\mathbb{Z}$ is finite, it is enough to show that this map is injective. Suppose that $[(N/2)x^2-cx]-[(N/2)y^2-cy]\in 2^t\mathbb{Z}$ with $x,y\in\mathbb{Z}$. Then $(x-y)[(N/2)(x+y)-c]\in 2^t\mathbb{Z}$. Since $(c,N)=1$, $2\nmid c$. Therefore $x-y\in 2^t\mathbb{Z}$ and this map is injective.
\end{proof}

It is obvious that $\tau_2(\sigma)=1$ when $\sigma\in2^{-1}\mathbb{Z}_2$. Now suppose that $\sigma=2^{-t}\alpha$ with $\alpha\in\mathbb{Z}_2^\times$ and $t\geq2$. Then we have
$$\tau_2(2^{-t}\alpha)=2^{1-t}\displaystyle\sum_{x\in\mathbb{Z}/2^{t-1}\mathbb{Z}}\mathbf{e}_2(2^{1-t}\alpha[(N/2)x^2-cx])=2^{1-t}\displaystyle\sum_{y\in\mathbb{Z}/2^{t-1}\mathbb{Z}}\mathbf{e}_2(2^{1-t}\alpha y)=0.
$$
Therefore $\tau_2(\sigma)$ is $0$ or $1$, depending on whether
$\sigma\notin 2^{-1}\mathbb{Z}_2$ or $\sigma\in 2^{-1}\mathbb{Z}_2$.

\begin{thm}
Suppose that $4\mid N$. Then
\[
b_2(h,\lambda,0)=2^{-(\mathrm{ord}_2N-1)}.
\]
\end{thm}
\begin{proof}
Calculating the same as in the proof of Theorem \ref{thm:pdivN}, we obtain
\begin{align*}
b_2(h,\lambda,0)&=\displaystyle\int_{\mathbb{Q}_2}\int_{V_2}\mathbf{e}_2(\sigma(\varphi(v)-h))\lambda_2(v)dvd\sigma\\
&=\displaystyle\int_{\mathbb{Q}_2}\int_{L_2}\mathbf{e}_2(\sigma(\varphi(v+\nu)-h))dvd\sigma\\
&=\displaystyle\int_{\mathbb{Q}_2}\mathbf{e}_2\left(\sigma\left(\sum_{j=1}^na_j(c/N)^2-h\right)\right)\prod_{j=1}^n\tau_2\left(\frac{a_j\sigma}N\right)d\sigma\\
&=2^{-(\mathrm{ord}_2N-1)}.\nonumber
\end{align*}
\end{proof}

\subsection{\texorpdfstring{$\mathbf{p=2}$}{p equals 2} and \texorpdfstring{$\mathbf{2\nmid N}$}{2 does not divide N}}\label{subsec5}
\vskip20pt
In this case, $\nu\in L_2$. Hence $X_2=L_2$ and we can apply \cite[Theorem 4.1]{Yang} to obtain bounds for local densities. Since $L_2$ is independent of $N$, the local density clearly only depends on the lattice $L$, but it remains to show that the density is positive, for which we need to restrict the possible choice of lattice.
Suppose that $L_2$ is equivalent to $\langle b_12^{r_1},...,b_n2^{r_n}\rangle$ with $b_i\in\mathbb{Z}_2^\times$ and $r_1\leq\cdots\leq r_n$.

For each integer $t>0$ we denote
\begin{align*}
L(t,1)&:=\{r_i : r_i-t<0\ \text{is odd}\}, &l(t,1)&:=\#L(t,1),\\
\varepsilon(t)&:=\prod_{i\in L(t-1,1)}b_i, &d(t)&:=t+\frac{1}{2}\sum_{r_i<t-1}(r_i-t+1),\\
\delta(t)&:=
  \begin{cases}
    0 & \text{if $r_i=t-1$ for some $i$,} \\
    1 & \text{otherwise.}
  \end{cases}
&\displaystyle\left(\frac{2}{x}\right)&:=
  \begin{cases}
    (2,x)_2, & \text{if }x\in\mathbb{Z}_2^\times, \\
    0, & \text{otherwise.}
  \end{cases}
\end{align*}
Furthermore, for $h=\alpha2^a$ with $\alpha\in\mathbb{Z}_2^\times$ and $a$ is a nonnegative integer. we define
\begin{multline*}
R_1(1,h,L_2):=\sum_{\substack {1<t\leq a+3\\ l(t-1,1)\ \text{is odd}}}\delta(t)\left(\frac{2}{\mu\varepsilon(t)}\right)2^{d(t)-3/2}\\
+\sum_{\substack {1<t\leq a+3\\ l(t-1,1)\ \text{is even}}}\delta(t)\left(\frac{2}{\varepsilon(t)}\right)2^{d(t)-1}\mathbf{e}_2\left(\frac{\mu}{8}\right)\mathrm{char}(4\mathbb{Z}_2)(\mu),
\end{multline*}
where $\mu=\mu_t(h)$ is given by $\mu_t(h):=\alpha2^{a+3-t}-\sum_{r_i<t-1}b_i$ and $\mathrm{char}(Y)$ stands for the characteristic function of a set $Y$. Then
Yang \cite{Yang} has shown that
\[
b_2(h,\lambda,0)=1+R_1(1,h, L_2).
\]

\vskip10pt
\begin{thm}\label{2 does not divide N}
Suppose that $n\geq6$ and 2 does not divide $N$. If $[r_1, r_2, r_3, r_4]$ is equal to one of the following:
\begin{enumerate}
\item $[0,0,0,i],\ 0\leq i\leq 2$;
\item $[0,0,1,i],\ 1\leq i\leq 3$;
\item $[0,1,1,i],\ 1\leq i\leq 2$;
\item $[0,1,2,i],\ 2\leq i\leq 3$.
\end{enumerate}
Then there is an absolute positive number $d(L)$ depends only on $L$ such that
$$d(L)\leq b_2(h,\lambda,0)\leq 2.$$
\end{thm}
\begin{proof}
The proof of Theorem \ref{2 does not divide N} is straightforward by combining the results of
the following two lemmas.
\end{proof}

\begin{lem}\label{lemma4}
Suppose that $n\geq6$ is even, 2 does not divide $N$ and $[r_1,r_2,r_3,r_4]$ satisfies one of conditions in Theorem \ref{2 does not divide N}. Then
$$\sum_{\substack {t\geq r_n+2 \\ l(t-1,1)\ \text{is odd}}}2^{d(t)-3/2}+\sum_{\substack {t\geq r_n+2\\ l(t-1,1)\ \text{is even}}}2^{d(t)-1}\leq2^{r_4-r_n-1}.$$
Furthermore, when $[r_1,r_2,r_3,r_4]=[0,0,0,2]$ or $[0,0,1,3]$, this upper bound can be improved to $2^{r_4-r_n-2}$.
\end{lem}
\begin{proof}
Let $s:=\frac{1}{2}(r_1+r_2+r_3+r_4)$. We first suppose that $n=6$ and compute the contribution for the terms $r_6+2\leq t<\infty$.  For the cases $[r_1,r_2,r_3,r_4]=[0,1,1,1]$ and $[0,1,2,2]$, if $r_5=r_6$, then $l(t-1,1)$ is odd and $d(t)-3/2\leq s-1/2-t\leq r_4-t$, while if $r_5<r_6$, then $d(t)-3/2< d(t)-1\leq s-1/2-t\leq r_4-t$. For the remaining cases, we have $s\leq r_4$, and hence $d(t)-3/2<d(t)-1\leq s-t\leq r_4-t$.

We next assume that $n\geq 8$ and $r_n+2\leq t<\infty$.
Since there are at least $8$ summands and $s\leq r_4+1/2$ in all cases satisfying the conditions of Theorem \ref{2 does not divide N}, we have $d(t)-3/2<d(t)-1\leq s-t-1\leq r_4-t$.

Combining the above cases, we conclude that
\[
\sum_{\substack {t\geq r_n+2\\ l(t-1,1)\ \text{is odd}}}2^{d(t)-3/2}+\sum_{\substack {t\geq r_n+2\\ l(t-1,1)\ \text{is even}}}2^{d(t)-1}\leq\sum_{t\geq r_n+2}2^{r_4-t}\leq2^{r_4-r_n-1}.
\]
Moreover, when $[r_1,r_2,r_3,r_4]=[0,0,0,2]$ or $[0,0,1,3]$ and $r_n+2\leq t<\infty$,
we have $d(t)-1\leq s-t\leq r_4-1-t$, which improves the above bound by a factor of $1/2$.
\end{proof}

\vskip10pt
\begin{lem}\label{lemma5}
Suppose that $n\geq6$, 2 does not divide $N$, $[r_1,r_2,r_3,r_4]$ satisfies one of conditions in Theorem \ref{2 does not divide N}, and $r_4+2\leq r_n$. Then
\[
\sum_{\substack {r_4+2\leq t\leq r_n\\ l(t-1,1)\ \text{is odd}}}\delta(t)2^{d(t)-3/2}+\sum_{\substack {r_4+2\leq t\leq r_n\\ l(t-1,1)\ \text{is even}}}\delta(t)2^{d(t)-1}\leq 2^{-1}+\cdots+2^{r_4-r_n+1}.
\]
Furthermore, when $[r_1,r_2,r_3,r_4]=[0,0,0,2]$ or $[0,0,1,3]$, we can improve this upper bound to $2^{-2}+\cdots+2^{r_4-r_n}$.
\end{lem}

\begin{proof}
Let $s:=\frac{1}{2}(r_1+r_2+r_3+r_4)$. Assume first that $r_4+2\leq t\leq r_5$. In this case, we have $d(t)=s-t+2$. For every choice satisfying the conditions of Theorem \ref{2 does not divide N}, if $l(t-1,1)$ is odd, we have $s\leq r_4+1/2$ and $d(t)-3/2\leq r_4+1-t$, while if $l(t-1,1)$ is even, we have $s\leq r_4$ and $d(t)-1\leq r_4+1-t$.

Next consider the case $r_5+2\leq t\leq r_n$. In this case there are at least $5$ summands in the sum defining $d(t)$, and therefore $d(t)-1\leq s-t+1/2\leq r_4+1-t$. Finally, when $[r_1,r_2,r_3,r_4]=[0,0,0,2]$ or $[0,0,1,3]$ and $r_4+2\leq t\leq r_n$, $d(t)-1\leq s-t+1\leq r_4-t$, which yields the improved bound in that case.
\end{proof}

\subsection{Upper Bounds and Lower Bounds} We conclude this section with the main theorems of this part.

\begin{theorem}\label{thm:EisensteinBounds}
Suppose that $L, \nu, h$ satisfy the conditions given at the beginning of this section. If $L$ also satisfies the conditions in Theorem \ref{p does not divide N} and Theorem \ref{2 does not divide N}, then there exist absolute positive constants $A(L)$ and $B(L)$ which depend only on $L$ such that
\[
A(L)\frac{h^{\frac{n-2}{2}-\varepsilon}}{N}\leq\frac{R(h,X)}{m(X)}\leq B(L)\frac{h^{\frac{n-2}{2}+\varepsilon}}{N}.
\]
\end{theorem}
\begin{proof}
It is enough to find bounds for the products $\prod_{p\mid e_1}b_p(h,\lambda,0)$ and $\prod_{p\mid he', p\nmid e_1}r_p\left(\frac{n}{2}\right)$ appearing in \eqref{eqn:ShimuraEis}.

Note that the values of $e'$ and $e_1$ depend only on $N$ and the structure of the
lattice $L$. We bound the product of local densities by combining the results obtained in Subsections \ref{subsec1}--\ref{subsec5}. To obtain a bound on the product of factors $r_p(n/2)$, let $\omega(x):=\sum_{p|x}1$
denote the number of distinct prime divisors of $x$. Robin \cite[Theorem 11]{Robin} has shown that $\omega(x)\leq 1.3841\log x(\log\log x)^{-1}$, and hence $2^{\omega(x)}\ll_{\varepsilon} x^{\varepsilon}$. Combining this with the fact that $1/2\leq r_p(n/2)\leq 2$ when $n\geq 6$ is even, we obtain the claim.
\end{proof}

\begin{remark}
When $p\mid2\det(L)$ but $p\nmid N$, we have $b_p(h,\lambda,0)\leq \max\{2,\mathrm{ord}_p\det(L)\}$. Then
\[
\frac{R(h,X)}{m(X)}\leq\frac{4\pi^{\frac{n}{2}}\zeta(\frac{n}{2})^2}{\Gamma(\frac{n}{2})\sqrt{\det(L)}}\cdot\prod_{p\mid 2\det(L)}\max\{2,\mathrm{ord}_p\det(L)\}\cdot\frac{h^{\frac{n-2}{2}+\varepsilon}}{N}\leq\frac{16\pi^{\frac{n}{2}}\zeta(\frac{n}{2})^2}{\Gamma(\frac{n}{2})}\cdot\frac{h^{\frac{n-2}{2}+\varepsilon}}{N}
\]
Hence we see that the constant
\begin{equation}\label{eqn:Bndef}
B(L)=:B_n
\end{equation}
 may be chosen to only depend on the rank $n$ of $L$, and not the individual lattice $L$.
\end{remark}

We can also obtain an upper bound for $L+\nu$ with a more general set of vectors $\nu$ (note that for a lower bound we would require some additional restriction on the lattice).

\begin{theorem}\label{thm:EisensteinUpperGen}
Suppose that $L$ satisfies the condition given at the beginning of this section, and that $L+\nu$ is a coset with conductor $N$ (i.e., $N\nu$ is not necessarily of the form $c(v_1+\dots+v_n)$). Then there exists an absolute positive constant $C(L)$ which depends only on $L$ such that
\[
\frac{R(h,X)}{m(X)}\leq C(L)\frac{h^{\frac{n-2}{2}+\varepsilon}}{N}.
\]
\end{theorem}

\begin{proof}
We first assume that $p \nmid N$. Since $d(t)$ is decreasing to $-\infty$ as $t$ goes to $\infty$, the local density $b_p(h, \lambda,0)$ is trivially bounded by some positive constant which only depends on $L$. We only need to recalculate the local densities $b_p(h, \lambda,0)$ where $p\mid N$. Suppose that $\nu=-(c_1v_1+\cdots+c_nv_n)/N$. As $L+\nu$ is a coset with conductor $N$, for each $p\mid N$ there exists some $c_i$ such that $(c_i, p)=1$. We then compute
\begin{multline*}
b_p(h,\lambda,0)=\displaystyle\int_{\mathbb{Q}_p}\int_{V_p}\mathbf{e}_p(\sigma(\varphi(v)-h))\lambda_p(v)dvd\sigma\\
=\displaystyle\int_{\mathbb{Q}_p}\int_{L_p}\mathbf{e}_p(\sigma(\varphi(v+\nu)-h))dvd\sigma\hspace{2.6in}\\
=\displaystyle\int_{\mathbb{Q}_p}\mathbf{e}_p\left(\sigma\left(\sum_{j=1}^na_j(c_j/N)^2-h\right)\right)\prod_{j=1}^n\int_{\mathbb{Z}_p}\mathbf{e}_p\left(\frac{\sigma a_j(2c_j, N)}{N}\left(\frac{N}{(2c_j, N)}x^2-\frac{2c_j}{(2c_j, N)}x\right)\right)dx d\sigma.
\end{multline*}
Note that for $j=1,...,n$
\[
\left|\int_{\mathbb{Z}_p}\mathbf{e}_p\left(\frac{\sigma a_j(2c_j, N)}{N}\left(\frac{N}{(2c_j, N)}x^2-\frac{2c_j}{(2c_j, N)}x\right)\right)dx\right|\leq 1.
\]
In particular, when $(c_i, p)=1$, by similar arguments as in Lemma \ref{p divides N} and Lemma \ref{2 parallel N} we have
\[
\int_{\mathbb{Z}_p}\mathbf{e}_p\left(\frac{\sigma a_i(2c_i, N)}{N}\left(\frac{N}{(2c_i, N)}x^2-\frac{2c_i}{(2c_i, N)}x\right)\right)dx=1
\]
if $\sigma\in p^{\mathrm{ord}_p N-\mathrm{ord}_p (2,N)-\mathrm{ord}_p a_i}\mathbb{Z}_p$ when $p\mid N/(2,N)$, or if $\sigma\in 2^{\mathrm{ord}_2 N-\mathrm{ord}_2 (2,N)-\mathrm{ord}_2 a_i-1}\mathbb{Z}_2$ when $2\| N$. Otherwise it is 0. Therefore, when $p\mid N$
\[
b_p(h, \lambda,0)\leq 4p^{\mathrm{ord}_p \det(L)-\mathrm{ord}_p N}.
\]
\end{proof}

\vskip20pt
\section{Mass Formula For Shifted Lattices}\label{sec:massformula}
In this section, we would like to find an upper bound for the mass for a shifted lattice $X=L+\nu$, where $L$ and $\nu$ satisfy the conditions given at the beginning of Section \ref{sec:Eisenstein}. Recall that the mass of $X$ is defined as
\[
m(X):=\sum_{i=1}^r|\O(X_i)|^{-1},
\]
 where $X_1,...,X_r$ are the representatives of the isometry classes in the genus of $X$. Now let $Y_1,...,Y_s$ be the representatives of the \textit{proper} isometry classes in the genus of $X$
and define
\[
m^+(X):=\sum_{j=1}^s|\SO(Y_j)|^{-1}.
\]
We define $m(K)$ and $m^+(K)$ analogously for any lattice $K$.

\begin{lemma}
We have
\[
m^+(X)=2m(X).
\]
\end{lemma}
\begin{proof}
For any $W$ in the genus of $X$, note that $[\O(W):\SO(W)]=2$ or $\O(W)=\SO(W)$, according to the cases where $\O(W)$ contains an isometry with determinant $-1$ or not. When $[\O(W):\SO(W)]=2$, we have $\mathrm{cls}(W)=\mathrm{cls}^+(W)$ and $|\SO(W)|^{-1}=2|\O(W)|^{-1}$. Now we suppose that $\O(W)=\SO(W)$ and $\O(V)=\SO(V)\cup \SO(V)\tau$. Then we have $\mathrm{cls}(W)=\mathrm{cls}^+(W)\cup\mathrm{cls}^+(\tau W)$ and $|\SO(W)|^{-1}+|\SO(\tau W)|^{-1}=2|\O(W)|^{-1}$.
\end{proof}

\begin{lemma} \label{bound of ratio}
Let $K$ be a lattice of rank $n\geq 4$ over $\Z$, and let $u_0$ be a primitive vector in $K_p$. Then
\[
\left[\SO(K_p):\SO\!\left(K_p+\frac{u_0}{p^t}\right)\right]\leq 2p^{2\ord_p(2\det K)}p^{(n-1)t}.
\]
\end{lemma}

\begin{proof}
We want to show that we can choose $\beta_p$ to be $2p^{2\ord_p(2\det K)}$ in \cite[lemma 4.1]{Sunliang}. The proof of \cite[lemma 4.1]{Sunliang} shows that
\[
[\SO(K_p):\SO(K_p+\frac{u_0}{p^t})]\leq |\overline{x}\in K_p/p^tK_p : \varphi(x)=\varphi(u_0)\pmod{p^t}|.
\]
Consider the quotient
\[
\frac{|\overline{x}\in K_p/p^tK_p : \varphi(x)=\varphi(u_0) \pmod{p^t}|}{p^{(n-1)t}}.
\]
By \cite[Hilfssatz 13]{Siegel}, when $t>2\ord_p(2\det K)$ this quotient is a constant which is the local representation density of $K_p$ at $\varphi(u_0)$. Applying Yang's formulas in Subsection \ref{subsec2} and Subsection \ref{subsec5}, we can bound the local density by $2p^{2\ord_p(2\det K)}$ for each $p$. When $t\leq2\ord_p(2\det K)$, the number of elements in $K_p/p^tK_p$ is less than or equal to $p^{nt}$, thus the quotient is bounded by $p^t\leq 2p^{2\ord_p(2\det K)}$.

\end{proof}

\begin{theorem}
Suppose that $X=L+\nu$ where $L$ and $\nu$ satisfy the conditions given at the beginning of Section \ref{sec:Eisenstein}. Suppose that the rank $n$ of $L$ is at least $4$. Then
\[
m(X)\ll_{n} \det(L)^{\frac{n+5}{2}} N^{n-1+\varepsilon}.
\]
\end{theorem}
\begin{proof}
Given a shifted lattice $Y$, we define
\[
\mathrm{Mass}^+(Y):=\sum_{W\in\mathrm{gen}^+(Y)}\frac{1}{|\SO(W)|},
\]
where the sum runs through the representatives of proper isometry classes in $\mathrm{gen}^+(Y)$. Note that $\mathrm{cls}^+(Y)\subseteq\mathrm{gen}^+(Y)\subseteq\mathrm{gen}(Y)$. If $X_1,...,X_t$ are the representatives of proper genera in the genus of $X$, then we have
\[
m^+(X)=\sum_{i=1}^t\mathrm{Mass}^+(X_i).
\]

Since $X_i$ is in the genus of $X$, there exists $\alpha_i\in \O_{\mathbb{A}}(V)$ for which $X_i=\alpha_i X$. This implies that $X_i=\alpha_iL+\nu_i$, where $\nu_i-(\alpha_i)_p\nu\in(\alpha_i)_pL_p$ for every finite prime $p$. Hence $\nu_i\in(\alpha_i)_pL_p$ for primes $p\nmid N$, and when $p\mid N$ we have $\nu_i=\frac{c}{N}w_i$ for some primitive vector $w_i\in (\alpha_i)_pL_p$.

Now we want to find an upper bound for $\mathrm{Mass}^+(X_i)$, $i=1,...,t$. Let $\tau$ be the Tamagawa measure on $O^+_{\mathbb{A}}(V)$ (for details, the reader may refer to \cite[Appendix B, Section 4]{Cassels}). Then we have
\[
m^+(\alpha_iL)=2\tau_{\infty}(\SO(V_\infty))^{-1}\prod_{p\neq \infty}\tau_p(\SO((\alpha_iL)_p))^{-1}.
\]
By \cite[Corollary 2.5]{Sunliang2016}, we have
\[
\mathrm{Mass}^+(X_i)=2\tau_{\infty}(\SO(V_\infty))^{-1}\prod_{p\neq \infty}\tau_p(\SO((X_i)_p))^{-1}.
\]
Note that all the $\mathrm{Mass}^+(X_i)$ are the same because all the local proper orthogonal groups $\SO((X_i)_p)$ are conjugate and the Tamagawa measure is a Haar measure.
Therefore,
\begin{align*}
\mathrm{Mass}^+(X_i)&=\displaystyle m^+(\alpha_iL)\prod_{p\mid N}[\SO((\alpha_iL)_p):\SO((X_i)_p)]\\
                    &=\displaystyle m^+(L)\prod_{p\mid N}[\SO((\alpha_iL)_p):\SO((X_i)_p)].
\end{align*}

Let $S$ be the Gram matrix of a lattice $L$. For every prime $p$ and positive integer $r$, let $A_{p^r}(L)$ denote the number of $n\times n$ integral matrices $T$ mod $p^r$ such that
\[
T^tST\equiv S \hspace{5mm}\pmod{p^r}.
\]
Then by \cite[Hilfssatz 13]{Siegel}, $\frac{1}{2}p^{-(r/2)n(n-1)}A_{p^r}(L)$ is independent of $r$ when $r$ is large enough, and the value will be denoted by $\alpha_p(L)$. By \cite[Theorem 6.8.1]{Kitaoka}, we have
\[
m^+(L)=2m(L)=4\pi^{-n(n+1)/4}\prod_{i=1}^n\Gamma(i/2)\det(L)^{\frac{n+1}{2}}\times\prod_p\alpha_p(L)^{-1}.
\]
In the following paragraph, we use \cite[Theorem 5.6.3]{Kitaoka} to calculate the local density $\alpha_p(L)$ for each prime $p$. Notice that the definition of local density is normalized differently in \cite{Kitaoka}; namely $\beta_p(L,L)=2^{-n\delta_{2,p}+1}\alpha_p(L)$, where $\delta_{2,p}$ is Kronecker's delta function.

Suppose that $L_p=\bot_jM_j$ where $M_j$ is either $p^j$-modular or $\{0\}$, and $M_j\neq\{0\}$ occurs only for finitely many integers $j$. Write $M_j=N_j^{(p^j)}$ for the scaling of the bilinear form on the unimodular lattice $N_j$ by $p^j$ and let $n_j:=\mathrm{rank}(N_j)=\mathrm{rank}(M_j)$. First we consider the case when $p\geq3$. Let $s$ be the number of nonzero components $M_j$. By \cite[Theorem 5.6.3]{Kitaoka} we have
\begin{align*}
\alpha_p(L)&\geq2^{s-1}p^{\sum_{j\geq 0} jn_j(n_j+1)/2}\prod_{j\geq0,M_j\neq 0}(1+p^{-n_j/2})^{-1}\prod_{i=1}^{\left\lfloor\frac{n_j}{2}\right\rfloor}\!\left(1-p^{-2i}\right)\\
           &\geq\prod_{i=1}^{\left\lfloor\frac{n_0}{2}\right\rfloor}\!\left(1-p^{-2i}\right)\cdot\left(1+p^{-n_0/2}\right)^{-1}\cdot\prod_{j>0,M_j\neq 0}2(1+p^{-n_j/2})^{-1}p^{jn_j(n_j+1)/2}\prod_{i=1}^{\left\lfloor\frac{n_j}{2}\right\rfloor}\!\left(1-p^{-2i}\right)\\
           &\geq\prod_{i=1}^{\left\lfloor\frac{n_0}{2}\right\rfloor}(1-p^{-2i})\cdot(1+p^{-n_0/2})^{-1}\cdot p^{s-1}\\
           &\geq\left(1-p^{-2}\right)\!\left(1-p^{-4}\right)\cdots\!\left(1-p^{-n+2}\right)\!\left(1-p^{-n/2}\right).
\end{align*}
Then we consider the case when $p=2$. For convenience, $N_j=\{0\}$ is also called even. We define
\[
t_j:=\left\{
                                                 \begin{array}{ll}
                                                   0, & \hbox{if $N_j$ is even;} \\ \\
                                                   n_j, & \hbox{if $N_j$ is odd and $N_{j+1}$ is even;} \\ \\
                                                   n_j+1, & \hbox{if $N_j$ and $N_{j+1}$ are odd.}
                                                 \end{array}
                                               \right.
\]
For a unimodular lattice $M$, we write $M=M(e)\perp M(o)$, where $M(e)$ is even and $M(o)$ is either odd or $\{0\}$ with $\mathrm{rank}(M(o))\leq 2$. For any $M_j\neq\{0\}$, we put
\[
F_j=\left\{
  \begin{array}{ll}
    \frac{1}{2}(1+2^{-\mathrm{rank}(N_j(e))/2}), & \hbox{if both $N_{j-1}$ and $N_{j+1}$ are even}\\
    & \hbox{and unless $N_j(o)\cong\langle\varepsilon_1\rangle\perp\langle\varepsilon_2\rangle$}\\
    & \hbox{with $\varepsilon_1\equiv\varepsilon_2 \pmod 4$;} \\ \\
    \frac{1}{2}, & \hbox{otherwise.}
  \end{array}
\right.
\]
Then by \cite[Theorem 5.6.3]{Kitaoka},
\begin{align*}
\alpha_2(L)&\geq2^{n-1}2^{\sum_{j\geq 0} jn_j(n_j+1)/2-\sum_{j\geq 0} t_j}\prod_{j\geq0,M_j\neq 0}F_j^{-1}\cdot \prod_{j\geq0,M_j\neq 0}\prod_{i=1}^{\left\lfloor\frac{n_j}{2}\right\rfloor}\!\left(1-2^{-2i}\right)\\
           &\geq 2^{n-1}2^{-\sum_{j\geq 0} t_j}\prod_{j\geq0,M_j\neq 0}F_j^{-1}\prod_{i=1}^{\lfloor\frac{n_0}{2}\rfloor}(1-2^{-2i})\cdot\prod_{j>0,M_j\neq 0}2^{jn_j(n_j+1)/2}\prod_{i=1}^{\lfloor\frac{n_j}{2}\rfloor}(1-2^{-2i})\\
           &= 2^{n-1}\prod_{j\geq0,M_j\neq 0}2^{-t_j}F_j^{-1}\prod_{i=1}^{\lfloor\frac{n_0}{2}\rfloor}(1-2^{-2i})\cdot \prod_{j>0,M_j\neq 0}2^{jn_j(n_j+1)/2}\prod_{i=1}^{\lfloor\frac{n_j}{2}\rfloor}\!\left(1-2^{-2i}\right)\\
           &\geq 2^{n-1}2^{-\sum_{j\geq 0} n_j}\prod_{i=1}^{\lfloor\frac{n_0}{2}\rfloor}(1-2^{-2i})\cdot \prod_{j>0,M_j\neq 0}2^{jn_j(n_j+1)/2}\prod_{i=1}^{\lfloor\frac{n_j}{2}\rfloor}\!\left(1-2^{-2i}\right)\\
           &\geq\frac{1}{2}\left(1-2^{-2}\right)\!\left(1-2^{-4}\right)\cdots\!\left(1-2^{-n+2}\right)\!\left(1-2^{-n/2}\right).
\end{align*}
Hence, we conclude that
\[
\alpha_p(L)\geq 2^{-\delta_{2,p}}\left(1-p^{-2}\right)\!\left(1-p^{-4}\right)\cdots\!\left(1-p^{-n+2}\right)\!\left(1-p^{-n/2}\right)
\]
for each prime $p$. Therefore
\[
m^+(L){\ll_n} \det(L)^{\frac{n+1}{2}}.
\]
Finally, we obtain that $[\SO((\alpha_iL)_p):\SO((X_i)_p)]\ll 2p^{2\ord_p(2\det L)} p^{(n-1)\mathrm{ord}_pN}$ by Lemma \ref{bound of ratio}.
Combining this with the above results, we can conclude that for $i=1,...,t$,
\[
\mathrm{Mass}^+(X_i)\ll_{n}\det(L)^{\frac{n+5}{2}}N^{n-1+\varepsilon}.
\]

Now it is enough to count the number of proper genera in the genus of $X$. Notice that
\[
\mathrm{gen}^+(X)=\{Y\in\mathrm{gen}(X) : Y_p\in \mathrm{cls}^+(X_p)\ \text{for any finite prime}\ p\}.
\]
If $p\nmid N$, then we have $\mathrm{cls}^+(X_p)=\mathrm{cls}^+(L_p)=\mathrm{cls}(L_p)=\mathrm{cls}(X_p)$, while for $p\mid N$ we have $1\leq [\mathrm{cls}(X_p):\mathrm{cls}^+(X_p)]\leq 2$. Therefore there are at most $2^{\omega(N)}=O(N^{\varepsilon})$ proper genera in the genus of $X$, where $\omega(N)$ is the number of prime divisors of $N$.
\end{proof}
\vskip20pt

\section{Bound On The Cuspidal Part}\label{sec:cusp}
The coefficients of cusp forms have attracted a great deal of attention. In \cite{Ramanujan}, Ramanujan studied the coefficients of $\Delta(z)$, the unique normalized cusp form of weight $12$ for $\SL_2(\Z)$ given by
\[
\Delta(z)=q\prod_{n=1}^\infty(1-q^n)^{24}=\sum_{n=1}^\infty\tau(n)q^n,
\]
where $q=e^{2\pi iz}$. Ramanujan conjectured that $\tau(n)\leq \sigma_0(n)n^{\frac{11}{2}}$, where $\sigma_0(n)$ is the number of positive divisors of $n$. Then Petersson \cite{Petersson} generalized Ramanujan's conjecture to cusp forms for congruence subgroups of $\SL_2(\Z)$. Eichler \cite{Eichler} first recognized the role of arithmetic geometry in relation to the Ramanujan-Petersson conjecture by reducing the weight $k=2$ case to the Weil conjectures for algebraic curves over finite fields. Shimura, Kuga, Ihara and Deligne similarly reduced this conjecture for all weights $k\geq 2$ to the full Weil conjectures (see \cite{Deligne1} for details). Finally, the Ramanujan-Petersson conjecture was proven by Deligne \cite{Deligne} as a consequence of his work on the Weil conjectures.

Deligne's result applies to newforms, certain cusp forms that are eigenforms for all of the Hecke operators. For such a weight $k$ newform, Deligne's work implies that the coefficient of $q^n$ is bounded above by $\sigma_0(n)n^{\frac{k-1}{2}}$. Any cusp form can be written as a linear combination of newforms and newforms acted on by various operators, so it is still the case that the coefficients of a general cusp form $f$ are $O(\sigma_0(n)n^{\frac{k-1}{2}})$. However, the implied constant depends heavily on $f$, and it is a nontrivial problem to determine this constant.

In this section, we will first study this implied constant and then we will use the above bound to obtain an upper bound on the coefficients of the cuspidal part. We use an explicit argument of Schulze-Pillot and Yenirce \cite{S-PY} in order to obtain a bound on the Fourier coefficients in terms of the Petersson norm.

\begin{lemma}\label{lem:normbound}
Let $\Gamma_{M,N}:=\Gamma_0(M N^2)\cap \Gamma_1(N)$ and $S_{k}\left(\Gamma_{M,N},\psi\right)$ be the space of cusp forms of weight $k$ for $\Gamma_{M,N}$ with Nebentypus $\psi$ in the usual sense that the multiplier for $\gamma=\left(\begin{smallmatrix} a&b\\ c&d\end{smallmatrix}\right)\in \Gamma_{M,N}$ is $\psi(d)$.
Suppose that $k\in \N$ and $f\in S_{k}\left(\Gamma_{M,N},\psi\right)$ for some character $\psi$.
\noindent

\noindent
\begin{enumerate}
\item
For every $h\in \N$
\begin{equation}\label{eqn:normboundineffective}
|a_f(h)|\ll_{\varepsilon,k} \|f\|_{\Gamma_{M,N}} h^{\frac{k-1}{2}+\varepsilon} (MN)^{\varepsilon}.
\end{equation}
The implied constant in \eqref{eqn:normboundineffective} is ineffective.

\item
For every $h\in \N$
\begin{equation}\label{eqn:normboundeffective}
|a_f(h)|\ll_{\varepsilon,k} \|f\|_{\Gamma_{M,N}} h^{\frac{k-1}{2}+\varepsilon} \!\left(MN^2\right)^{\frac{1}{2}+\varepsilon}.
\end{equation}
The implied constant in \eqref{eqn:normboundeffective} is effective.
\end{enumerate}
\end{lemma}
\begin{remarks}
\noindent

\noindent
\begin{enumerate}
\item
The implied constant in \eqref{eqn:normboundineffective} is ineffective due to an ineffective bound for the reciprocal of the norm of a newform due to Hoffstein and Lockhart \cite{HL} (actually, they prove the result for the Maass form case and note in the remark at the bottom of \cite[page 164]{HL} that the same methods yields the result for holomorphic modular forms). This has been made effective in a few special cases (for example, see the work of Rouse \cite{Rouse}). The effective bound \eqref{eqn:normboundeffective} uses an effective version due to Fomenko \cite{Fomenko}.
\item The result in Lemma \ref{lem:normbound} may essentially be read off from \cite[Theorem 11]{S-PY}. However, they use the normalization $\|f\|$ instead of $\|f\|_{\Gamma}$, and hence we provide a full argument for the reader in order to avoid confusion.

\end{enumerate}
\end{remarks}

\begin{proof}
(1) By \cite[Theorem 2.5]{Cho}, we have
\begin{equation}\label{eqn:Cho}
S_{k}\!\left(\Gamma_{M,N}, \psi\right)=\bigoplus_{\chi} S_{k}\!\left(\Gamma_0\!\left(MN^2\right),\chi\psi\right),
\end{equation}
where the sum runs over all Dirichlet characters modulo $N$; recall that there are $\varphi(N)$ such Dirichlet characters.

Let $H_k^{\operatorname{new}}(\ell,\chi \psi)$ denote the set of normalized newforms of weight $k$ and level $\ell$ with Nebentypus $\chi\psi$.
For $g_1\neq g_2 \in H_{k}^{\operatorname{new}}(\ell,\chi\psi)$ it is well known that $\left<g_1,g_2\right>=0$.
Moreover, for each $g\in H_{k}^{\operatorname{new}}(\ell,\chi\psi)$, Schulze-Pillot and Yenirce \cite{S-PY} constructed an explicit orthogonal basis of the eigenspace $W_g$ spanned by $g|V_d$ with $d\mid \frac{MN}{\ell}^2$. They then proved (see the proof of \cite[Theorem 11]{S-PY}) that there exists an orthonormal (with respect to $\left<\cdot,\cdot\right>$, not $\left<\cdot,\cdot\right>_{\Gamma}$) basis $\{F_{g,d} : d\mid \frac{MN^2}{\ell}\}$ for $W_g$ such that
\[
\left|a_{F_{g,d}}(h)\right|\leq \frac{\sigma_0(h) h^{\frac{k-1}{2}}}{\|g\|}d^{\frac{1}{2}}\prod_{p\mid d}\left(1+\frac{1}{p}\right)^2
\]
We hence write
\begin{equation}\label{eqn:fexpand}
f=\sum_{\chi (\operatorname{mod}N)}\sum_{\ell\mid MN^2}\sum_{g\in H_k^{\operatorname{new}}(\ell,\chi\psi)}\sum_{d\mid \frac{MN^2}{\ell}} \alpha_{g,d} F_{g,d}
\end{equation}
and note that
\begin{equation}
\label{eqn:fnormGamma}
\|f\|_{\Gamma_{M,N}}^2=\left[\SL_2(\Z):\Gamma_{M,N}\right]\sum_{\chi (\operatorname{mod}N)}\sum_{\ell\mid MN^2}\sum_{g\in H_k^{\operatorname{new}}(\ell,\chi\psi)}\sum_{d\mid \frac{MN^2}{\ell}} \left|\alpha_{g,d}\right|^2.
\end{equation}
We then bound
\begin{multline}\label{eqn:boundf1}
|a_f(h)|\leq \sum_{\chi (\operatorname{mod}N)}\sum_{\ell\mid MN^2}\sum_{g\in H_k^{\operatorname{new}}(\ell,\chi\psi)}\sum_{d\mid \frac{MN^2}{\ell}} \left|\alpha_{g,d}\right|
\left|a_{F_{g,d}}(h)\right|
\\
\leq \sigma_0(h)h^{\frac{k-1}{2}} \sum_{\chi (\operatorname{mod}N)}\sum_{\ell\mid MN^2}\sum_{g\in H_k^{\operatorname{new}}(\ell,\chi\psi)}\sum_{d\mid \frac{MN^2}{\ell}} \frac{\left|\alpha_{g,d}\right|}{\|g\|} d^{\frac{1}{2}}\prod_{p\mid d} \left(1+\frac{1}{p}\right)^2.
\end{multline}

We now apply Cauchy-Schwartz and \eqref{eqn:fnormGamma} to bound \eqref{eqn:boundf1} as
\begin{equation}\label{eqn:boundf2}
|a_f(h)|\leq  \frac{\sigma_0(h)h^{\frac{k-1}{2}}\|f\|_{\Gamma_{M,N}}}{\left[\SL_2(\Z):\Gamma_{M,N}\right]^{\frac{1}{2}}}\left(\sum_{\chi (\operatorname{mod}N)}\sum_{\ell\mid MN^2}\sum_{g\in H_k^{\operatorname{new}}(\ell,\chi\psi)}\sum_{d\mid \frac{MN^2}{\ell}} \frac{d}{\|g\|^2}\prod_{p\mid d} \left(1+\frac{1}{p}\right)^4\right)^{\frac{1}{2}}.
\end{equation}
Noting that $\prod_{p\mid d}\left(1+\frac{1}{p}\right)^4\ll d^{\varepsilon}\ll_{\varepsilon} (MN)^{\varepsilon}$, trivially bounding $d\leq \frac{MN^2}{\ell}$, and using the formula \eqref{eqn:Gamma0index} yields
\begin{multline}\label{eqn:small2}
\mathcal{N}_{M,N}:=\sum_{\chi (\operatorname{mod}N)}\sum_{\ell\mid MN^2}\sum_{g\in H_k^{\operatorname{new}}(\ell,\chi\psi)}\sum_{d\mid \frac{MN^2}{\ell}} \frac{d}{\|g\|^2}\prod_{p\mid d} \left(1+\frac{1}{p}\right)^4\\
\ll_{\varepsilon} \left(MN^2\right)^{1+\varepsilon}\sum_{\chi (\operatorname{mod}N)}\sum_{\ell\mid MN^2}\ell^{-1}\sum_{g\in H_k^{\operatorname{new}}(\ell,\chi\psi)}\frac{1}{\|g\|^2}.
\end{multline}

We then use the bound for the reciprocal of the norm given by Hoffstein and Lockhart \cite{HL}, namely
 % (cf. \cite[(3.4)]{BM} for the statement with this normalization of the inner product, and note that a lower bound with the same power of $\ell$, up to the $\varepsilon$ has been found by Iwaniec \cite{Iwaniec})
\[
\|g\|_{\Gamma_0(\ell)}^{-2}\ll_{\varepsilon} \frac{(4\pi)^{k-1}}{\ell \Gamma(k)} (k\ell)^{\varepsilon}\ll_{\varepsilon,k} \ell^{\varepsilon-1},
\]

Since
\begin{equation}\label{eqn:primsize}
\# H_{k}^{\operatorname{new}}(\ell,\chi\psi)\leq \dim_{\C}\!\left(S_{k}(\ell,\chi\psi)\right)\leq \frac{k}{12}\left[\SL_2(\Z):\Gamma_0(\ell)\right]\ll_{k,\varepsilon} \ell^{1+\varepsilon}
\end{equation}
 by the valence formula, and the fact that there are $\varphi(N)$ characters $\chi$ yields that
\begin{equation}\label{eqn:small3}
\mathcal{N}_{M,N}\ll_{\varepsilon,k} \left(MN^2\right)^{1+\varepsilon} \varphi(N)\ll MN^2\varphi(N)(MN)^{\varepsilon}.
\end{equation}
We finally evaluate the index $\left[\SL_2(\Z):\Gamma_{M,N}\right]$. Since $\Gamma_{M,N}$ is precisely the kernel of the map $\varphi:\Gamma_0(MN^2)\to \Z/N\Z$ given by $\varphi\!\left(\!\left(\begin{smallmatrix}a &b \\ c &d\end{smallmatrix}\right)\right):=a+N\Z$ and $a$ may run through all choices relatively prime to $N$ (since $c\equiv 0\pmod{N}$, $a$ must be relatively prime), we have
\begin{equation}\label{eqn:indexmiddle}
\left[\Gamma_0\!\left(MN^2\right):\Gamma_{M,N}\right]=\varphi(N).
\end{equation}
 Hence by \eqref{eqn:Gamma0index} we have
\[
\left[\SL_2(\Z):\Gamma_{M,N}\right]=\left[\SL_2(\Z):\Gamma_0\!\left(MN^2\right)\right]\left[\Gamma_0\!\left(MN^2\right):\Gamma_{M,N}\right]=MN^2\varphi(N) \prod_{p\mid MN^2}\!\left(1+\frac{1}{p}\right),
\]
 We therefore obtain from \eqref{eqn:small3} that $\mathcal{N}_{M,N}\ll_{k,\varepsilon}\left[\SL_2(\Z):\Gamma_{M,N}\right](MN)^{\varepsilon}$, and plugging this back into \eqref{eqn:boundf2} yields the claim.

(2) We return to \eqref{eqn:small2}. Instead of using Hoffstein and Lockhart's bound for the reciprocal of the norm, however, we instead use a bound of Fomenko \cite{Fomenko}. Namely, we have (this bound is used in \cite[Theorem 11]{S-PY}, but with the normalization $\|g\|$).
\[
\|g\|_{\Gamma_0(\ell)}^2\geq 4\pi e^{4\pi},
\]
so that \eqref{eqn:small2} and \eqref{eqn:primsize} imply that
\begin{multline*}
\mathcal{N}_{M,N}\ll_{\varepsilon} \left(MN^2\right)^{1+\varepsilon}\sum_{\chi (\operatorname{mod}N)}\sum_{\ell\mid MN^2}\ell^{-1}\left[\SL_2(\Z):\Gamma_{0}(\ell)\right]\sum_{g\in H_k^{\operatorname{new}}(\ell,\chi\psi)}1\\
\ll_{\varepsilon,k} \left(MN^2\right)^{1+\varepsilon}\sum_{\chi (\operatorname{mod}N)}\sum_{\ell\mid MN^2}\ell^{1+\varepsilon}\ll \left(MN^2\right)^{2}\varphi(N)(MN)^{\varepsilon}=MN^2\left[\SL_2(\Z):\Gamma_{M,N}\right] (MN)^{\varepsilon}.
\end{multline*}
\end{proof}

We are now ready to obtain a bound for the coefficients of the cuspidal part of the theta function for $L+\nu$.
\begin{proposition}\label{prop:cusppart}
Suppose that $L$, $\nu$ and $h$ satisfy the conditions given at the beginning of Section \ref{sec:Eisenstein}. If the rank $n$ of $L$ is even and $n\geq6$, then the number of representations $r_{L+\nu}(h)$ satisfies
\[
r_{L+\nu}(h)= a_{E_{L+\nu}}(h) +O_{\varepsilon,n}\left(N_{L}^{n+\frac{7}{2}+\varepsilon}N^{\frac{3n}{2}+\frac{5}{2}+\varepsilon}h^{\frac{\frac{n}{2}-1}{2}+\varepsilon} \right),
\]
where $a_{E_{L+\nu}}(h)$ is the product of the local densities and $N_L$ is the level of $L$.
\end{proposition}

\begin{remark}
A similar result holds for $n$ odd, except that the power of $h$ is worse, although it is conjectured that the same power of $h$ holds for odd $n\geq 5$. Implied constants in Proposition \ref{prop:cusppart} are absolute (independent of $m$) and effectively computable, but not explicit.
\end{remark}

\begin{proof}[Proof of Proposition \ref{prop:cusppart}]
We generalize an argument of Duke from \cite{DukeTernary}.  Although Duke was interested in the ternary quadratic form case, the argument can be applied more generally.

We write the associated theta function $\Theta_{L+\nu}(z):=\theta_{L+\nu}(N^2z)$ and split
\[
\Theta_{L+\nu}(z)=E_{L+\nu}(N^2z)+G_{L+\nu}(N^2z),
\]
where $E_{L+\nu}$ is the Eisenstein series component and $G_{L+\nu}$ is the cuspidal component. Letting $g_{L+\nu}(z):=G_{L+\nu}(N^2z)$, we have
\[
a_{\Theta_{L+\nu}}(N^2h)=a_{E_{L+\nu}}(h) + a_{g_{L+\nu}}(N^2h).
\]
Let $h_N:=N^2h$. Then it suffices to prove that
\[
\left|a_{g_{L+\nu}}\left(h_N\right)\right|\ll_{\varepsilon,n}N_{L}^{n+\frac{7}{2}+\varepsilon} N^{\frac{3n}{2}+\frac{5}{2}+\varepsilon} h^{\frac{\frac{n}{2}-1}{2}+\varepsilon}.
\]
By \cite[Theorem 2.4]{Cho}, $g_{L+\nu}$ is a weight $n/2$ cusp form on $\Gamma_{N_L,N}$, and hence Lemma \ref{lem:normbound} (2)
implies that we have the effective bound
\begin{equation}\label{eqn:Deligne}
\left|a_{g_{L+\nu}}(h_N)\right| \ll_{\varepsilon,n}N_L^{\frac{1}{2}}N
\left(N_{L}N\right)^{\varepsilon} \left\|g_{L+\nu}\right\|_{\Gamma_{N_L,N}}h_N^{\frac{\frac{n}{2}-1}{2}+\varepsilon},
\end{equation}
where $\|g_{L+\nu}\|_{\Gamma_{N_L,N}}^2$ is the Petersson norm of $g_{L+\nu}$ normalized as in \eqref{eqn:innerdef}.  It remains to bound the norm.

Closely following the argument in \cite[Lemma 1]{DukeTernary}, we claim that for any cusp form $f$ of integral weight $k\geq 2$ on $\Gamma\supseteq \Gamma(M)$ with cusp width $N_{i\infty}$ at $i\infty$, we have
\begin{equation}\label{eqn:norm}
\|f\|_{\Gamma}^2 \ll \Gamma(\alpha)M^{3\alpha+3-3k} N_{i\infty}^{\alpha} \left[\SL_2(\Z):\Gamma\right] \sum_{\ell=1}^{\infty}\left|a_{f}(\ell)\right|^2 \ell^{-\alpha}
\end{equation}
for any $\alpha>k-1$ for which the sum converges.  Directly from the definition \eqref{eqn:innerdef} of the inner product, for a cusp form $f$ on a subgroup $\Gamma\subseteq\SL_2(\Z)$, denoting $z=x+iy$, we have
\begin{equation}\label{eqn:normdef}
\|f\|_{\Gamma}^2=\sum_{\gamma\in \Gamma\backslash \SL_2(\Z)}\int_{\mathcal{F}} |f(\gamma z)|^2 \im(\gamma z)^{k}\frac{dx dy}{y^2}.
\end{equation}
Here
\[
\mathcal{F}:=\left\{ z\in \H: |x|\leq \frac{1}{2},\ |z|\geq 1\right\}
\]
denotes the standard fundamental domain in the upper half plane $\H$. Let $N_{\sigma}$ denote the cusp width at $\sigma$. Since the integrand in \eqref{eqn:normdef} is non-negative, we may bound from above against the integral over the box
\[
\left\{z\in\H:  |x| \leq \frac{1}{2},\ y\geq \frac{\sqrt{3}}{2}\right\}.
\]
Namely, we obtain (choosing $\gamma_\sigma\in\SL_2(\Z)$ such that $\gamma_\sigma(i\infty)=\sigma$)
\begin{multline}\label{eqn:innertriv}
\displaystyle\sum_{\substack{\gamma\in \Gamma\backslash\SL_2(\Z)\\ \gamma(i\infty)=\sigma}}\int_{\mathcal{F}} |f(\gamma z)|^2 \im(\gamma z)^{k}\frac{dx dy}{y^2}
=\displaystyle\sum_{n=0}^{N_{\sigma}-1}\int_{\mathcal{F}} |f(\gamma_\sigma(z+n))|^2 \im(\gamma_\sigma (z+n))^{k}\frac{dx dy}{y^2}\\
\ll \displaystyle\int_{\frac{\sqrt{3}}{2}}^{\infty} \int_{-\frac{N_{\sigma}}{2}}^{\frac{N_{\sigma}}{2}} |f(\gamma_\sigma z)|^2 \im(\gamma_\sigma z)^{k}\frac{dx dy}{y^2}.
\end{multline}
Due to the exponential decay of $f$ at the cusps, we may choose (as in \cite[(10)]{DukeTernary}) $C_{\sigma}:=\sqrt{3}/2+C N_{\sigma}$ (with $C$ an absolute constant independent of $f$) such that
\[
\int_{C_{\sigma}}^{\infty}\int_{-\frac{N_{\sigma}}{2}}^{\frac{N_{\sigma}}{2}}|f(\gamma_\sigma z)|^2 \im(\gamma_\sigma z)^k \frac{dx dy}{y^2}\leq \int_{\frac{\sqrt{3}}{2}}^{C_{\sigma}} \int_{-\frac{N_{\sigma}}{2}}^{\frac{N_{\sigma}}{2}}|f(\gamma_\sigma z)|^2 \im(\gamma_\sigma z)^k \frac{dx dy}{y^2},
\]
and thus we may bound \eqref{eqn:innertriv} by
\begin{equation}\label{eqn:innertriv2}
\ll\int_{\frac{\sqrt{3}}{2}}^{C_{\sigma}} \int_{-\frac{N_{\sigma}}{2}}^{\frac{N_{\sigma}}{2}} |f(\gamma_\sigma z)|^2 \im(\gamma_\sigma z)^{k}\frac{dx dy}{y^2}.
\end{equation}
Plugging back into \eqref{eqn:normdef} yields
\[
\|f\|_{\Gamma}^2 \ll \sum_{\sigma} \int_{\frac{\sqrt{3}}{2}}^{C_{\sigma}} \int_{-\frac{N_{\sigma}}{2}}^{\frac{N_{\sigma}}{2}} |f(\gamma_\sigma z)|^2 \im(\gamma_\sigma z)^{k}\frac{dx dy}{y^2}.
\]
When $\sigma=i\infty$, we take $\gamma_{i\infty}=\left(\begin{smallmatrix} 1&0\\ 0&1\end{smallmatrix}\right)$. Then $\im(\gamma_{i\infty} z)=\im z$ and $\re(\gamma_{i\infty} z)=\re z$. Now writing $\gamma_\sigma=\left(\begin{smallmatrix} a&b\\ c&d\end{smallmatrix}\right)$ with $\sigma=a/c$, we have
\begin{align*}
\im(\gamma_\sigma z)&=\frac{y}{|j(\gamma_\sigma,z)|^2}=\frac{y}{c^2[(x+d/c)^2+y^2]} \gg c^{-2}N_{\sigma}^{-2},\\
\left|\re(\gamma_\sigma z)-\sigma\right|&=\left|\frac{a}{c}-\frac{cx+d}{c|j(\gamma_\sigma,z)|^2} - \frac{a}{c}\right|= \frac{\left|x+\frac{d}{c}\right|}{|j(\gamma_\sigma,z)|^2}\ll 1\ll N_{i\infty},
\end{align*}
where $j(\gamma_\sigma,z):=cz+d$.  Hence the change of variables $z\mapsto \gamma_\sigma^{-1}z$ in the integral in \eqref{eqn:innertriv2} yields
\[
\|f\|_{\Gamma}^2 \ll \sum_{\sigma\neq i\infty} \int_{c^{-2}N_{\sigma}^{-2}}^{\infty} \int_{-\frac{N_{i\infty}}{2}}^{\frac{N_{i\infty}}{2}} |f(z)|^2 y^{k}\frac{dx dy}{y^2}+ \int_1^{\infty} \int_{-\frac{N_{i\infty}}{2}}^{\frac{N_{i\infty}}{2}} |f(z)|^2 y^{k}\frac{dx dy}{y^2}.
\]
We may now plug in the Fourier expansion of $f$ at $i\infty$; namely $f(z)=\sum_{\ell\geq 1} a_f(\ell)e^{2\pi i \ell z/N_{i\infty}}$.  This yields
\begin{multline*}
\|f\|_{\Gamma}^2 \ll \sum_{\sigma\neq i\infty} \int_{c^{-2}N_{\sigma}^{-2}}^{\infty} \sum_{\ell_1=1}^{\infty}\sum_{\ell_2=1}^{\infty}a_{f}(\ell_1)\overline{a_{f}(\ell_2)} e^{-\frac{2\pi(\ell_1+\ell_2)y}{N_{i\infty}}} y^{k-2} dy \int_{-\frac{N_{i\infty}}{2}}^{\frac{N_{i\infty}}{2}} e^{\frac{2\pi i (\ell_1-\ell_2)x}{{N_{i\infty}}}} dx\\
+ \int_{1}^{\infty} \sum_{\ell_1=1}^{\infty}\sum_{\ell_2=1}^{\infty}a_{f}(\ell_1)\overline{a_{f}(\ell_2)} e^{-\frac{2\pi(\ell_1+\ell_2)y}{N_{i\infty}}} y^{k-2} dy \int_{-\frac{N_{i\infty}}{2}}^{\frac{N_{i\infty}}{2}} e^{\frac{2\pi i (\ell_1-\ell_2)x}{{N_{i\infty}}}} dx\\
=N_{i\infty}\sum_{\sigma\neq i\infty} \sum_{\ell=1}^{\infty} \left|a_{f}(\ell)\right|^2 \int_{c^{-2}N_{\sigma}^{-2}}^{\infty}  e^{-\frac{4\pi \ell  y}{N_{i\infty}}} y^{k-1} \frac{dy}{y}+ N_{i\infty} \sum_{\ell=1}^{\infty} \left|a_{f}(\ell)\right|^2 \int_{1}^{\infty}  e^{-\frac{4\pi \ell  y}{N_{i\infty}}} y^{k-1} \frac{dy}{y}\\
=N_{i\infty}^k\sum_{\sigma\neq i\infty} \sum_{\ell=1}^{\infty} \left|a_{f}(\ell)\right|^2(4\pi \ell)^{1-k} \!\int_{\frac{4\pi \ell}{c^{2}N_{\sigma}^2 N_{i\infty}}}^{\infty}\!\!  e^{-y} y^{k-1} \frac{dy}{y}+ N_{i\infty}^k \sum_{\ell=1}^{\infty} \left|a_{f}(\ell)\right|^2(4\pi \ell)^{1-k} \!\int_{\frac{4\pi \ell}{N_{i\infty}}}^{\infty}\!  e^{-y} y^{k-1} \frac{dy}{y}.
\end{multline*}
We now slightly alter the argument from \cite{DukeTernary}.  We note that the remaining integrals may be expressed in terms of the \begin{it}incomplete gamma function\end{it}
\[
\Gamma(s,y):=\int_{y}^{\infty} e^{-t} t^{s-1} dt.
\]
We thus have
\[
\|f\|_{\Gamma}^2\ll N_{i\infty}^k\sum_{\ell=1}^{\infty} \left|a_{f}(\ell)\right|^2(4\pi \ell)^{1-k} \left(\Gamma\left(k-1,\frac{4\pi \ell}{N_{i\infty}}\right)+ \sum_{\sigma\neq i\infty} \Gamma\left(k-1,\frac{4\pi \ell}{c^2 N_{\sigma}^2N_{i\infty}}\right) \right).
\]
Since $s=k-1$ is a positive integer, we can evaluate the incomplete gamma function as (cf. \cite[8.8.9]{NIST})
\[
\Gamma(s,y) = (s-1)!e^{-y}\sum_{j=0}^{s-1} \frac{y^{j}}{j!}.
\]
We thus obtain
\[
\|f\|_{\Gamma}^2\ll \sum_{j=0}^{k-2}\frac{1}{j!} N_{i\infty}^{k-j}\sum_{\ell=1}^{\infty} \left|a_{f}(\ell)\right|^2(4\pi \ell)^{1-k+j} \left(e^{-\frac{4\pi \ell}{N_{i\infty}}}+ \sum_{\sigma\neq i\infty} c^{-2j}N_{\sigma}^{-2j}e^{-\frac{4\pi \ell}{c^2 N_{\sigma}^2N_{i\infty}}} \right).
\]
Note that for $\sigma\neq i\infty$, we may choose $c\mid M$ and $\sigma=\frac{a}{c}$ with $0\leq a<c$ (we also have $1\leq N_{\sigma}\leq M$).  Therefore, for an integer $k\geq 2$, we have
\[
\|f\|_{\Gamma}^2\ll \sum_{j=0}^{k-2}\frac{1}{j!} N_{i\infty}^{k-j}\sum_{\ell=1}^{\infty} \left|a_{f}(\ell)\right|^2(4\pi \ell)^{1-k+j} \left(e^{-\frac{4\pi \ell}{N_{i\infty}}}+ \sum_{c\mid M}\sum_{a(\operatorname{mod} c)} c^{-2j}N_{a/c}^{-2j} e^{-\frac{4\pi \ell}{c^2 N_{a/c}^2 N_{i\infty}}} \right)
%\ll \sum_{j=0}^{k-2}\frac{1}{j!} N_{i\infty}^{k-j}\sum_{\ell=1}^{\infty} \left|a_{f}(\ell)\right|^2(4\pi \ell)^{1-k+j} \left(e^{-\frac{4\pi \ell}{N_{i\infty}}}+ \sum_{c\mid M} c^{1-2j} e^{-\frac{4\pi \ell}{c^2 M N_{i\infty}}} \right).
\]
Suppose next that we have a bound
\begin{equation}\label{eqn:coeffbnd}
\left|a_f(\ell)\right|\leq C_f(4\pi \ell)^{r},
\end{equation}
for some $C_f,r\in\R$ (we shall determine the constants $C_f$ and $r$ for our particular usage in \eqref{eqn:Crbounds} below). Then we obtain
\begin{equation}\label{eqn:toint}
{\arraycolsep 2pt
\begin{array}{rcl}
\|f\|_{\Gamma}^2&\ll&\displaystyle C_f^{2}\sum_{j=0}^{k-2}\frac{1}{j!} N_{i\infty}^{k-j}\sum_{\ell=1}^{\infty} (4\pi \ell)^{2r+ 1-k+j} \left(e^{-\frac{4\pi \ell}{N_{i\infty}}}+ \sum_{\sigma\neq i\infty} c^{-2j}N_{\sigma}^{-2j} e^{-\frac{4\pi \ell}{c^2 N_{\sigma}^2 N_{i\infty}}} \right)\\
&\ll&\displaystyle C_f^{2}N_{i\infty}^{2r+1}\sum_{j=0}^{k-2}\frac{1}{j!} \sum_{\ell=1}^{\infty}\left(\frac{4\pi \ell}{N_{i\infty}}\right)^{2r+ 1-k+j} e^{-\frac{4\pi \ell}{N_{i\infty}}}\\
& &+\displaystyle C_f^2N_{i\infty}^{2r+1}\sum_{j=0}^{k-2}\frac{1}{j!}\sum_{\sigma\neq i\infty} (c N_{\sigma}) ^{4r+2-2k}\sum_{\ell=1}^{\infty}\left(\frac{4\pi \ell}{c^2N_{\sigma}^2N_{i\infty}}\right)^{2r+1-k+j} e^{-\frac{4\pi \ell}{c^2 N_{\sigma}^2 N_{i\infty}}}.
\end{array}}
\end{equation}
We next bound the sums over $\ell$ in \eqref{eqn:toint} against integrals, giving
\begin{multline}\label{eqn:gammabound}
\|f\|_{\Gamma}^2\ll \displaystyle C_f^{2}N_{i\infty}^{2r+2}\sum_{j=0}^{k-2}\frac{1}{j!}\left(\int_{\frac{4\pi}{N_{i\infty}}}^{\infty} t^{2r+ 1-k+j} e^{-t}dt+\sum_{\sigma\neq i\infty} (cN_{\sigma})^{4r+4-2k} \int_{\frac{4\pi}{c^2 N_{\sigma}^2 N_{i\infty}}}^{\infty}t^{2r+1-k+j} e^{-t}dt\right)\\
=\displaystyle C_f^{2}N_{i\infty}^{2r+2}\sum_{j=0}^{k-2}\left(\frac{\Gamma\left(2r+2-k+j,\frac{4\pi}{N_{i\infty}}\right)}{j!}
+\sum_{\sigma\neq i\infty} (cN_{\sigma})^{4r+4-2k} \frac{\Gamma\left(2r+2-k+j, \frac{4\pi}{c^2 N_{\sigma}^2 N_{i\infty}}\right)}{j!}\right).
\end{multline}
When $2r+2-k+j> 0$, we can bound the incomplete gamma function against the gamma function (note that this holds for all $j$ if $r>\frac{k}{2}-1$). Otherwise, for $\re(s)<0$, we use the well-known asymptotic for $y\to 0$ (for example, rearrange \cite[8.8.9]{NIST} or \cite[8.4.15]{NIST} for the special case $s\in -\N$)
\[
\Gamma(s,y)\sim -\frac{y^s}{s}.
\]
Set
\[
\mathcal{C}_{\Gamma}:=\max\left\{ cN_{\sigma}: \sigma\text{ is a cusp of }\Gamma\right\}
\]
Note that Lemma \ref{lem:normbound} implies that any $r>\frac{k-1}{2}$ satisfies \eqref{eqn:coeffbnd} (although we see later that one must in practice choose $r$ larger in order to obtain an explicit bound for the constant $C_f$ in \eqref{eqn:coeffbnd}). Choosing some $r>\frac{k-1}{2}>\frac{k}{2}-1$ and bounding the incomplete gamma functions against gamma functions, we conclude that
\begin{equation}\label{eqn:norm1}
\|f\|_{\Gamma}^2\ll C_f^2N_{i\infty}^{2r+2} M \mathcal{C}_{\Gamma}^{4r+3-2k}\left[\SL_2(\Z):\Gamma\right]
\end{equation}
from \eqref{eqn:gammabound} together with $\sum_{\sigma} N_{\sigma}=[\SL_2(\Z):\Gamma]$ and the fact that $c\mid M$.

Recall that $g_{L+\nu}(z)=\Theta_{L+\nu}(z)-E_{L+\nu}(N^2z)$ is a modular form of weight $k=n/2$ on $\Gamma=\Gamma_0(N_L N^2)\cap \Gamma_1(N)$ with the same Nebentypus $\chi_L$ as $L$ (see \cite[Proposition 2.1]{Shimura} or \cite[Theorem 2.4]{Cho} for an explicit proof of this result). It is our goal to use \eqref{eqn:norm1} to obtain a bound for $\|g_{L+\nu}\|^2$. In order to do so, we use \eqref{eqn:Cho} to split $g_{L+\nu}$ into a sum
\[
g_{L+\nu}=\sum_{\chi} g_{\chi}
\]
with (see the proof of \cite[Theorem 2.5]{Cho})
\begin{equation}\label{eqn:fchidef}
g_{\chi}=g_{L+\nu,\chi}:=\frac{1}{\varphi(N)} \sum_{d\in (\Z/N\Z)^{\times}}\chi(d)^{-1} g_{L+\nu}\big|_{k,\chi_L} \gamma_d\in S_k\left(\Gamma_0(N_{L}N^2),\chi\chi_L\right),
\end{equation}
where $\gamma_d$ is any fixed element of $\Gamma_0(N_LN^2)$ with lower-right entry $d_0\equiv d \pmod{N}$.  However, \cite[Proposition 2.1]{Shimura} states that $\Theta_{L+\nu}$ maps to a theta series $\Theta_{L+\mu}$ associated to another shifted lattice $L+\mu$ of conductor $N$ (indeed, it moreover specifies that $\mu=a\nu$, where $a$ is the inverse of $d$ modulo $N$). Noting that the subspace spanned by Eisenstein series and the space of cusp forms are both preserved under the slash action, one may take the cuspidal parts of each side to conclude that $g_{L+\nu}|_{k,\chi_L}\gamma_d$ is the cusp form $g_{L+\mu}$ associated to $L+\mu$.

We then apply \eqref{eqn:norm1} with $f=g_{\chi}$. We have $M\ll N_LN^2$, $N_{i\infty}=1$, and
$\left[\SL_2(\Z):\Gamma\right]\ll N_L^{1+\varepsilon} N^{2+\varepsilon}.$ As pointed out by Duke, since the cusp width $N_{\sigma}$ for $\sigma=\frac{a}{c}$ in $\Gamma_0(M)$ is $M/(c^2,M)\leq M/c$ and $c\mid M$, we have $\mathcal{C}_{\Gamma_0(M)}\leq M$. Hence $\mathcal{C}_{\Gamma_0(N_{L}N^2)} \leq N_L N^2$ and we have
\[
\|g_{\chi}\|_{\Gamma_0\!\left(N_LN^2\right)}^2\ll C_{g_{\chi}}^2 N_{L}^{4r-n+5+\varepsilon}N^{8r-2n+10+\varepsilon}.
\]
Furthermore, since $a_{G_{L+\mu}}(\ell)=a_{\theta_{L+\mu}}(\ell)-a_{E_{L+\mu}}(\ell)$ and $a_{\theta_{L+\mu}}(\ell)\leq |O(L+\mu)|m(L+\mu)a_{E_{L+\mu}}(\ell)$, we have
\[
a_{g_L+\mu}(\ell_N)=a_{G_{L+\mu}}(\ell)\ll N_L^{(n+5)/2}N^{n-1+\varepsilon}a_{E_{L+\mu}}(\ell).
\]
Hence by \eqref{eqn:fchidef} and Theorem \ref{thm:EisensteinBounds} (together with the independence of $B(L)$ on $L$ observed in \eqref{eqn:Bndef})
\begin{equation}\label{eqn:Crbounds}
a_{g_{\chi}}(\ell_N) \ll B_n N_L^{(n+5)/2}\ell_N^{\frac{n}{2}-1+\varepsilon}.
\end{equation}
We may therefore choose $r=\frac{n}{2}-1+\varepsilon$ and $C_{g_{\chi}}\ll B_nN_L^{(n+5)/2}$ in \eqref{eqn:coeffbnd}. This yields
\[
\|g_{\chi}\|_{\Gamma_0\!\left(N_LN^2\right)}^2\ll B_n^2N_{L}^{2n+6+\varepsilon}N^{2n+2+\varepsilon}.
\]
We then plug back into \eqref{eqn:Deligne}, multiply by the square-root of the index from \eqref{eqn:indexmiddle} due to the change in the subgroup under which the inner product is being taken, and recall that there are $\varphi(N)\ll N $ characters $\chi$ to obtain
\begin{align*}
\left|a_{g_{L+\nu}}(h_N)\right|&\ll\sum_{\chi} |a_{g_{\chi}}(h_N)|\ll B_nN_{L}^{n+\frac{7}{2}+\varepsilon}N^{n+2+\varepsilon}\varphi(N)^{\frac{3}{2}}
h_N^{\frac{\frac{n}{2}-1}{2}+\varepsilon}\\
&\ll B_nN_{L}^{n+\frac{7}{2}+\varepsilon}N^{n+\frac{7}{2}+\varepsilon}h_N^{\frac{\frac{n}{2}-1}{2}+\varepsilon}.
\end{align*}
Thus we obtain, plugging back in $h_N=N^2h$,
\[
\left|a_{G_{L+\nu}}(h)\right|=\left|a_{g_{L+\nu}}(h_N)\right|\ll B_nN_{L}^{n+\frac{7}{2}+\varepsilon}N^{\frac{3n}{2}+\frac{5}{2}+\varepsilon}h^{\frac{\frac{n}{2}-1}{2}+\varepsilon}.
\]
\end{proof}

\section{Universal sums of polygonal numbers and the proof of Theorem \ref{thm:main}}\label{sec:universal}

In this section, we combine the results from Section \ref{sec:cusp} with Section \ref{sec:Eisenstein} to prove Theorem \ref{thm:main}.
\begin{proof}[Proof of Theorem \ref{thm:main}]
(1)  As the value of $\gamma_m$ has been determined when $m=3,4,5,6,8$, we assume that $m=7,9,10,\dots$. We construct escalator trees of generalized $m$-gonal sums up to depth $n_0=4$.  We note that there are no leaves of these partial trees by \cite[Theorem 1.1]{Sun}.  Specifically, for $m\geq 12$ these trees up to depth 4 are independent of $m$ and given by:
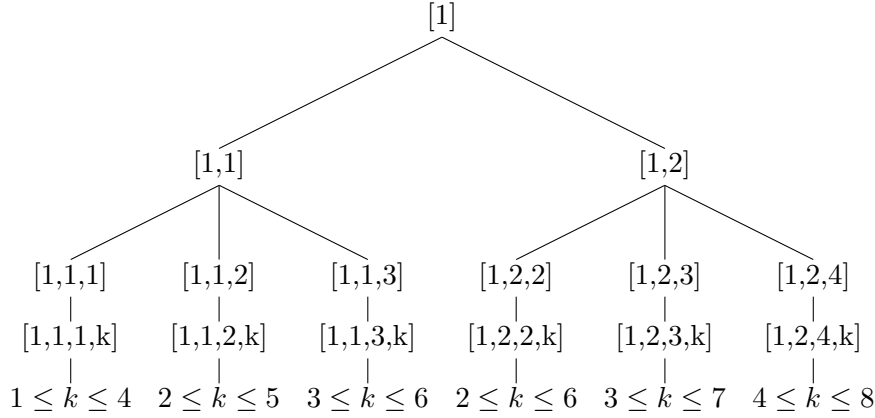
\[
\Tree
 [.[1] [.[1,1] [.[1,1,1] [.[1,1,1,k] $1\leq k\leq 4$ ]]  [.[1,1,2] [.[1,1,2,k] $2\leq k\leq 5$ ]] [.[1,1,3] [.[1,1,3,k] $3\leq k\leq 6$ ]]] [.[1,2] [.[1,2,2] [.[1,2,2,k] $2\leq k\leq 6$ ]]  [.[1,2,3] [.[1,2,3,k] $3\leq k\leq 7$ ]] [.[1,2,4] [.[1,2,4,k] $4\leq k\leq 8$ ]]]]
\]
For $m=7,9,10,11$, we can also obtain an explicit tree up to depth 4 (one can refer to \cite{Sun} for the truants), although the tree depends on the value of $m$. Therefore, for each even integer $n\geq 6$, there are only finitely many escalations of rank $n$, and we can check that these escalations
satisfy the conditions in Theorem \ref{thm:EisensteinBounds}.

In the following argument, we assume that $n\geq6$ is an even integer. For each $m$ and each choice of $a_{j,m}=[a_{j,m}^1,...,a_{j,m}^n]$ in the escalator tree, we find a bound $h_j(m)$ such that for $h>h_j(m)$, the corresponding shifted lattice $L_{j,m}+\nu_{j,m}$ represents $h$. As there are only finitely many $j$, we obtain
\[
\gamma_m\leq \max_{j} h_j(m).
\]
One key observation is that for $m$ sufficiently large (depending on $n$), the nodes $a_{j,m}$ are independent of $m$. For small $m$, we may bound $h_j(m)$ against a constant depending only on $n$, so we may assume that $a_{j,m}=a_j$ is independent of $m$, or in other words that the corresponding shifted lattice is $L_{j} +\nu_{j,m}$. The dependence of $h_j(m)$ therefore only lies in the shift $\nu_{j,m}$. This allows one to ignore the level $N_{L_j}$ of the lattice when computing the bound for $h_j(m)$, leading to a uniform bound depending only on $m$ and $n$.

We next show how to obtain the bound $h_j(m)$.  For each shifted lattice $L_{j,m}+\nu_{j,m}$ in the tree, the corresponding theta function $\theta_{L_{j,m}+\nu_{j,m}}$ counting the number of representations of $h$ by $L_{j,m}+\nu_{j,m}$ is a weight $n/2$ holomorphic modular form on some congruence subgroup $\Gamma$; the precise subgroup can be computed using \cite[Proposition 2.1]{Shimura}.

We then split
\[
\theta_{L_{j,m}+\nu_{j,m}}=E_{j,m}+G_{j,m},
\]
where $E_{j,m}$ is a weight $n/2$ Eisenstein series on $\Gamma$ and $G_{j,m}$ a weight $n/2$ cusp form on $\Gamma$.  By the Siegel--Weil theorem for shifted lattices (see \cite{Shimura2004} or \cite{vanderBlij}), the coefficients of $E_{j,m}$ are the products of local densities and simultaneously the weighted sums of numbers of representations by shifted lattices in the genus of $L_{j,m}+\nu_{j,m}$.

By Theorem \ref{thm:EisensteinBounds}, we have
\[
a_{E_{j,m}}(h)\gg_{\varepsilon,n}  m^{-1} h^{\frac{n}{2}-1-\varepsilon},
\]
while Proposition \ref{prop:cusppart} implies that
\[
a_{G_{j,m}}(h)\ll_{\varepsilon,n}m^{\frac{3n}{2}+\frac{5}{2}+\varepsilon} h^{\frac{\frac{n}{2}-1}{2}+\varepsilon}.
\]
Combining these, we see that the $h$th coefficient of $\theta_{L_{j,m}+\nu_{j,m}}$ is positive for
\[
h\gg_{\varepsilon,n}m^{6 + \frac{26}{n-2}+\varepsilon}.
\]
Finally, by \eqref{eqn:hdef}, we conclude that $\ell$ is represented by $f$ if
\[
\ell \gg m^{7+\frac{26}{n-2}+\varepsilon}.
\]
We next take $n=n_{\varepsilon}$ sufficiently large such that $\frac{26}{n-2}<\varepsilon$ to obtain the claim; note that since this is independent of $m$, the implied dependence on the lattice only depends on $\varepsilon$. However, in order to take $n$ sufficiently large, it is necessary that (for $m$ sufficiently large) none of the nodes of depth $n<n_{\varepsilon}$ are universal. Since $P_m(x)$ only takes the values $0$, $1$ and integers greater than or equal to $m-3$, there are at most $2^n$ integers less than $m-3$ which are represented by the sum \eqref{eqn:sumsmgonal}. Hence if $2^n<m-4$, we see that \eqref{eqn:sumsmgonal} is not universal. We conclude that the minimal depth of a leaf in the tree is depth at least $\log(m-4)$, which goes to infinity as $m\to \infty$. Hence for $m$ sufficiently large the depth of a leaf is larger than $n_{\varepsilon}$ in particular, and we conclude the claim.

(2)  Although this argument is contained in \cite{Guy}, we provide it for the convenience of the reader. Note that the number $1\leq \ell\leq m-4$ is not represented as a sum of $\ell-1$ generalized $m$-gonal numbers.  Since every integer is the sum of at most $m$ $m$-gonal numbers (as conjectured by Fermat and proven by \cite{Cauchy}), the form $\sum_{j=\ell}^{\ell+m-1} (\ell+1)P_m(x_j)$ represents every positive integer divisible by $\ell+1$, and thus
\[
\sum_{j=1}^{\ell-1} P_{m}(x_j) + \sum_{j=\ell}^{\ell+m-1} (\ell+1)P_m(x_j) + (2\ell+1)P_{m}(x_{\ell+m})
\]
represents every positive integer other than $\ell$.  We therefore obtain that $\gamma_{m}\geq \ell$.  Since $m$ is arbitrary and the only restriction is $\ell\leq m-4$, we see that for every $\ell\in\N$ there exists a sum of polygonal numbers which represents precisely every nonnegative integer other than $\ell$, and thus there is no uniform bound.
\end{proof}

\begin{remark}
Note that Guy was mainly interested in the smallest dimension for which the sum of polygonal numbers with $a_j=1$ is universal.  In other words, in our setting, Guy obtained a lower bound for the depth of the escalator tree, showing that the depth is at least $m-4$.   His result, combined with the fact that every positive integer is the sum of at most $m$ generalized $m$-gonal numbers implies that the actual depth under this particular restriction is between $m-4$ and $m$, giving a very tight bound.  There exist examples where Guy's lower bound is sharp; for example, Sun \cite{SunLagrange} has shown that every integer is the sum of four generalized octagonal numbers. Moreover, it was seen in the proof of Theorem \ref{thm:main} (1) that the minimal depth of a universal sum (i.e., a leaf of the tree) is at least $\log(m-4)$. It might be interesting to investigate the depth of other branches of the escalator tree.
\end{remark}

\section*{acknowledgments}
The authors would like to express their gratitude to Dr. Yuk-Kam Lau for his warm help and to the referees for their valuable comments.

\end{document}